\documentclass[12pt]{amsart}
\usepackage{array}
\usepackage{amsfonts}
\usepackage{amscd}
\usepackage{amssymb}
\usepackage{amsthm}
\usepackage{amsmath}
\usepackage{stmaryrd}
\usepackage{graphicx}
\usepackage{color}
\usepackage{verbatim}
\usepackage{mathrsfs}
\usepackage{tikz}
\usepackage{multirow}
\usetikzlibrary{calc}
\usepackage{calrsfs}
\usepackage{caption}
\usepackage[neveradjust]{paralist}

\parskip\medskipamount

\colorlet{ggrey}{black!50}

\usepackage{hyperref}
\usepackage[capitalize]{cleveref}

 \theoremstyle{plain}

\newtheorem{thm}{Theorem}[section]
\newtheorem{theo}{Theorem}

\newtheorem{lemma}[thm]{Lemma}
\newtheorem{prop}[thm]{Proposition}

\theoremstyle{definition}

\newtheorem{remark}[thm]{Remark}

\newtheorem{example}[thm]{Example}

\numberwithin{equation}{section}

\def\sA{\mathsf{A}}

\def\sD{\mathsf{D}}
\def\sE{\mathsf{E}}
\def\sF{\mathsf{F}}

\def\cI{\mathcal{I}}

\def\cL{\mathcal{L}}

\def \cP{\mathcal{P}}

\def\cS{\mathcal{S}}

\def\FF{\mathbb{F}}

\def\HH{\mathbb{H}}

\def\KK{\mathbb{K}}
\def\LL{\mathbb{L}}

\def\K{\mathbb{K}}
\def\PG{\mathsf{PG}}

\def\kar{\mathsf{char}}

\DeclareMathOperator\Res{\mathsf{Res}}

\def\Aut{\mathsf{Aut}}
\DeclareMathOperator\disp{\mathsf{Disp}}

\def\id{\mathsf{id}}
\def\<{\langle}
\def\>{\rangle}

\makeatletter
\renewcommand{\@makefnmark}{\mbox{\textsuperscript{}}}
\makeatother

\title{Automorphisms of Lie incidence geometries with spectral gaps}
\author{Yannick Neyt \and James Parkinson 
\and
Hendrik Van Maldeghem}
\date{\today}
\thanks{The first and third author are supported by the Fund for Scientific Research, Flanders, through the project G023121N}
\begin{document}

\begin{abstract} An automorphism of a building is called \textit{uniclass} if the Weyl distance between any chamber and its image lies in a unique (twisted) conjugacy class of the Coxeter group. In a previous paper we characterised uniclass automorphisms of spherical buildings in terms of their fixed structure. In the present paper we restrict to the simply laced case and characterise uniclass automorphisms in terms of a spectral gap property. More precisely, we show that an automorphism of a thick irreducible spherical building of simply laced type is uniclass if and only if no point of the long root subgroup geometry is mapped to distance~$1$ or codistance~$1$. 
%
%
\end{abstract}

\maketitle


\section*{Introduction}

Long root subgroup geometries of spherical buildings play a central role in the theory of group actions on buildings, following from the seminal work of Timmesfeld \cite{Tim:01}, and they also have strong connections with graded Lie algebras, as shown by Cohen \& Ivanyos \cite{Coh-Iva:06,Coh-Iva:07}. In the present paper we classify, in the simply laced case, automorphisms of such geometries that do not map any point to either a collinear one, or to a non-opposite point that is collinear to an opposite point (in other words, no point is mapped to either distance~$1$ or codistance~$1$). A primary motivation for this classification lies in the fact that this provides an elegant characterisation of so called \emph{uniclass automorphisms} of spherical buildings of simply laced type. We provide some more details below. 

Let $(\Delta,\delta)$ be a building with Coxeter system~$(W,S)$ and Coxeter graph $\Pi$. Let $\theta$ be an automorphism of $(\Delta,\delta)$, and let $\sigma\in\Aut(\Pi)$ be the \textit{companion automorphism}, defined by $\delta(C,D)=s$ if and only if $\delta(C^{\theta},D^{\theta})=s^{\sigma}$, for $s\in S$. The \textit{displacement spectra} of $\theta$ is
$$
\disp(\theta)=\{\delta(C,C^{\theta})\mid C\in\Delta\}.
$$
If $\Delta$ is thin then necessarily $\disp(\theta)$ is equal to a single $\sigma$-conjugacy class in~$W$. For a general building the displacement spectra $\disp(\theta)$ is contained in a union of $\sigma$-conjugacy classes in~$W$. We call $\theta$ \textit{uniclass} if $\disp(\theta)$ is contained in a single $\sigma$-conjugacy class. 

In \cite{NPV} we characterised the uniclass automorphisms of thick irreducible spherical buildings, showing that an automorphism of such a building is uniclass if and only if it is either \textit{anisotropic} (that is, maps all chambers to opposite chambers), or fixes a certain special subgeometry (called a \textit{Weyl substructure}). See Theorem~\ref{thm:basictheorem} for more details. 

The aim of the present paper is to provide an alternative characterisation of uniclass automorphisms, in the simply laced case, via a condition on the points of the associated long root subgroup geometries of the corresponding buildings. Since the diagrams are simply laced, all root subgroups are long root subgroups, and this ultimately makes such a characterisation possible. 

Recall that the long root subgroup geometries are point-line geometries which are non-strong parapolar spaces (for a precise definition, see Section~\ref{pre}) of diameter $3$. As such, the possibilities for the mutual position between two points are as follows:
\begin{compactenum}
\item[$\bullet$] the points are equal (called \textit{distance~$0$});
\item[$\bullet$] the points distinct and collinear (called \textit{distance $1$});
\item[$\bullet$] the points are not collinear, but are contained in a convex subspace isomorphic to a polar space of rank at least 2 (called \textit{distance~$2$}); 
\item[$\bullet$] the points $p,q$ are not collinear, but there exists a unique point collinear to both of them (called \textit{distance $2'$}, or \textit{codistance~$1$});
\item[$\bullet$] the points have mutual distance $3$ in the collinearity graph.
\end{compactenum}

Let $D\subseteq\{0,1,2,2',3\}$. An automorphism $\theta$ of the long root subgroup geometry of an irreducible thick spherical building with simply laced Coxeter diagram is called a \textit{$D$-kangaroo} if, for all points $x$, the mutual position between $x$ and its image $x^\theta$ is not contained in~$D$ (less formally, the mutual positions $d\in D$ are ``skipped'' in the displacement spectrum).

In this paper the $\{1,2'\}$-kangaroos play a special role -- these are automorphisms where both distance $1$ and codistance~$1$ are skipped in the displacement spectrum. Our main result is as follows. 

\begin{theo}\label{thm:A} An automorphism $\theta$ of the long root subgroup geometry of an irreducible thick spherical building with simply laced Coxeter diagram is uniclass if, and only if, it is a $\{1,2'\}$-kangaroo.
\end{theo}

The proof of Theorem~\ref{thm:A} requires some case-by-case analysis of types~$\mathsf{A}$, $\mathsf{D}$, and $\mathsf{E}$. In each instance we will work with concrete geometric models for the particular type of building, as described in Section~\ref{pre}. In Section~\ref{sec:counter} we provide examples in non-simply laced buildings illustrating that Theorem~\ref{thm:A} does not hold in this broader context.

As a byproduct of our methods we also obtain a classification of $\{2,2'\}$-kangaroos in long root subgroup geometries for simply laced types (see Section~\ref{sec:4}). Moreover we classify so-called linewise $\{2,2'\}$-kangaroos (see \cref{linewise} for the definition) for all polar spaces (buildings of Coxeter type $\mathsf{B}_n$ or $\mathsf{D}_n$; see Proposition~\ref{prop:22'polar}) and the $\{2,2'\}$-kangaroos in metasymplectic spaces (buildings of Coxeter type $\mathsf{F_4}$; see \cref{metasymp}).

\textbf{Structure of the paper---}In \cref{pre} we provide background on long root geometries (projective spaces, polar spaces and parapolar spaces), and we outline the theory of opposition and fix diagrams for automorphisms that will be required in our proofs. In \cref{sec:proofs} we prove \cref{thm:A} for the classical buildings (types $\mathsf{A}$ and $\mathsf{D}$). We also improve on an alternative characterisation of $\{1,2'\}$-kangaroos in hyperbolic polar spaces (type $\mathsf{D}$ buildings) given in \cite{NPVV}.
In \cref{sec:exc}, we prove \cref{thm:A} for the exceptional buildings (type $\mathsf{E}_n$, $n=6,7,8$). The basic strategy here is to first show that $\{1,2'\}$-kangaroos in type $\mathsf{E}_6$ are domestic (in the case of dualities), and in type $\mathsf{E}_7$ and $\mathsf{E}_8$ they are domestic and do not fix any chamber (see Theorem~\ref{thm:En}). This observation allows us to make use of the detailed classification of domestic automorphisms given in \cite{NPVV,PVMexc4,Mal:12}. Finally, in \cref{sec:4} we give a classification of $\{2,2'\}$-kangaroos in the simply laced cases and metasymplectic spaces, and also for \textit{linewise $\{2,2'\}$-kangaroos} for all polar spaces (see \cref{sec:4} for definitions).

\section{Preliminaries}\label{pre}

\subsection{Geometries from buildings} Let $\Delta$ be a thick building of spherical type~$\mathsf{X}_n$. We shall adopt Bourbaki labelling \cite{Bou:00} for Dynkin diagrams. Associated to $\Delta$ there are various point-line geometries (called \textit{Lie incidence geometries}). Let $J$ be a subset of the type set, and take the set of flags of type $J$ as the point set of a point-line geometry, with the line set determined by the panels of cotype $i$ with $i\in J$. Such a geometry is said to be a \textit{Lie incidence geometry of type $\mathsf{X}_{n,J}$} (if $J=\{j\}$ we write $\mathsf{X}_{n,J}=\mathsf{X}_{n,j}$). When the diagram of $\mathsf{X}_n$ is simply laced, the building $\Delta$ is uniquely determined by the diagram and a (possibly skew) field $\K$, in which case we denote the Lie incidence geometry of type $\mathsf{X}_{n,J}$ by $\mathsf{X}_{n,J}(\K)$. Each vertex of $\Delta$ has an interpretation in the Lie incidence geometry, usually as a singular subspace, or a symplecton, or another convex subspace. We introduce these notions below (they are based on the fact that Lie incidence geometries we consider are either projective spaces, polar spaces or parapolar spaces). We provide a brief introduction, but refer the reader to the literature for more background (e.g.~\cite{Shult}). 

All point-line geometries that we will encounter are \emph{partial linear spaces}, that is, two distinct points are contained in at most one common line---and points that are contained in a common line are called \emph{collinear}; a point on a line is sometimes also called \emph{incident with that line}. If two points $p,q$ are collinear (denoted as $p\perp q$), then we write $pq$ for the unique line containing them. We will also always assume that each line has at least three points. 

In a general point-line geometry $\Gamma=(X,\cL)$, where $X$ is the point set, and $\cL$ is the set of lines (which we consider here as a subset of the power set of $X$), one defines a \emph{subspace} as a set of points with the property that it contains all points of each line having at least two points with it in common. A subspace is called (a) \emph{singular} if each pair of points of it is collinear, (b) a \emph{hyperplane} if every line intersects it in at least one point (and then the line is either contained in it, or intersects it in exactly one point), and (c) \emph{convex} if all points of every shortest path between two members of the subspace are contained in the subspace. In (c), the shortest paths are taken in the \emph{collinearity graph}  (the graph with vertices the points, with vertices adjacent if the corresponding points are collinear and distinct).

\subsection{Projective spaces (buildings of type $\mathsf{A}$)} For a skew field $\K$, the \emph{projective space $\mathsf{A}_{n,1}(\K)$} is the point-line geometry with point set the $1$-spaces of an $(n+1)$-dimensional vector space over $\K$ (the \emph{underlying vector space}), and a typical line is the set of $1$-spaces contained in a $2$-space.  The family of singular subspaces is in one-to-one correspondence with the vertices of the building $\mathsf{A}_n(\K)$. An automorphism of a building of type $\mathsf{A}_n$ is either a \emph{collineation} of the corresponding projective space, that is, a permutation of the point set preserving the line set, or a \emph{duality}, that is, a bijection from the point set to the set of hyperplanes such that three collinear points are mapped onto three hyperplanes with pairwise the same intersection. Collineations and dualities induce a permutation of all subspaces. A duality acting on the set of subspaces as an order $2$ permutation is called a \emph{polarity}. 

\subsection{Polar spaces (buildings of type $\mathsf{B}$ or $\mathsf{D}$)} A polar space, for our purposes, is just a Lie incidence geometry of type $\mathsf{B}_{n,1}$ or $\mathsf{D}_{n,1}$. There is an axiomatic approach in which the main axiom is the so-called \emph{one-or-all axiom} due to Buekenhout \& Shult \cite{Bue-Shu:74}: 
\begin{itemize}
\item[(BS)] For every point $p$ and every line $L$, either each point on $L$ or exactly one point on $L$ is collinear to $p$. 
\end{itemize}
We also require that no point is collinear to all other points, and, to ensure finite rank, that each nested sequence of singular subspaces is finite. Then there exists a natural number $r$ such that each maximal singular subspace is a projective space of dimension $r-1$. Singular subspaces of dimension $r-2$ are called \emph{submaximal}.  We call $r$ the \emph{rank} of $\Gamma$. We allow rank 1, in which case we just have a geometry without lines (and we assume at least three points).

We will only work with some special type of polar spaces: a \emph{hyperbolic} polar space, which is related to a  building of type $\mathsf{D}_n$, $n\geq 4$, and arises from a nondegenerate quadric  of maximal Witt index in odd dimensional projective space.  It has the characterising property that every submaximal singular subspace is contained in precisely two maximal singular subspaces. 

\subsection{Parapolar spaces} A \emph{parapolar space} is a point-line geometry with connected incidence graph such that \begin{compactenum}
\item[(1)] each pair of noncollinear points are either (a) collinear with no common point, (b) collinear with exactly one common point, or (c) are contained in a convex subspace isomorphic to a polar space (called a \emph{symplecton}), and 
\item[(2)] each line is contained in a symplecton.\end{compactenum} 
We also require that there are at least two symplecta, and hence, the geometry is not a polar space. A pair of noncollinear points collinear to a unique common point is called a \emph{special} pair; a pair of noncollinear points contained in a common symplecton is called a \emph{symplectic} pair. Almost all Lie incidence geometries which are not projective or polar spaces turn out to be parapolar spaces. 

A special type of parapolar space occurs when we take the vertices of so-called \emph{polar type} of an irreducible spherical building as points. The polar type is the set of simple roots not perpendicular to the highest root (the polar type is a singleton set if the Dynkin diagram is not of type $\mathsf{A}_n$). Such a Lie incidence geometry is often called a \emph{long root subgroup geometry}. We will only need those of type $\mathsf{A}_{n,\{1,n\}}$, $n\geq 2$, $\mathsf{D}_{n,2}$, $n\geq 4$, $\mathsf{E_{6,2}}$, $\mathsf{E_{7,1}}$ and $\mathsf{E_{8,8}}$. In such parapolar spaces, point pairs are either identical, collinear, symplectic, special or opposite (the latter in the building theoretic sense -- such points have distance $3$ in the collinearity graph of the point-line geometry).  In such geometries, we have the notion of an \emph{equator} of two opposite points $p,q$, which is the set of points symplectic to both $p$ and~$q$. This set is turned into a geometry by letting the lines be defined by the symplecta through $p$ containing a given maximal singular subspace (maximal in both the symplecta and the whole geometry), and it is called the \emph{equator geometry}, denoted by $E(p,q)$. The equator geometry is isomorphic to the disjoint union of the long root subgroup geometry of the direct factors of the residue of a point (in the building theoretic sense). It is also a \emph{fully embedded} subgeometry, that is, the point set forms a subspace.  It is also an  \emph{isometric} embedding, that is, each pair of points is collinear, symplectic and special in the embedded geometry if, and only if, it is collinear, symplectic and special, respectively, in the ambient geometry. The points $p$ and $q$ are called \emph{poles} of $E(p,q)$ (they are not unique). 

Let $\Gamma=(X,\cL)$ be a Lie incidence geometry and let $U$ be a non-maximal singular subspace.  Then $U$ corresponds to a certain flag of the corresponding spherical building and we have a building theoretic notion of \emph{residue at $U$}. This is usually a reducible building. However,  in the geometry $\Gamma$ we distinguish the components of that residue and we 
define the \emph{upper residue at $U$} 
as the point-line geometry with point set the set of singular subspaces of dimension $\dim U+1$ containing $U$, where a typical line is formed by those singular subspaces containing $U$ that are contained in a given singular subspace of dimension $\dim U+2$ containing $U$. It is again a Lie incidence geometry (possibly corresponding to a reducible spherical building). 


\subsection{Properties} We now list some well-known properties of long root subgroup geometries. These can be found in \cite{Coh-Iva:06} and \cite{Shult}.

\begin{lemma}\label{pointlinesymp}
If in a long root subgroup geometry a point $p$ is collinear to at least one point of a symp $\xi$, and $\xi$ contains a point special to $p$, then $p$ is collinear precisely to a line $L$ of $\xi$.  All points of $\xi$ collinear to a unique point of $L$ are special to $p$. 
\end{lemma}

\begin{lemma}\label{specialspecial}
Let $p,u,w,q$ be four points of a long root subgroup geometry such that $p\perp u\perp w\perp q$. Then $p$ is opposite $q$ if, and only if, both pairs $\{p,w\}$ and $\{q,u\}$ are special. 
\end{lemma}

\begin{lemma}\label{pointsymp}
If, in a long root subgroup geometry, a symp $\xi$ contains a point $q$ opposite a given point $p$, then $\xi$ contains a unique point $x$ symplectic to $p$. All points of $\xi$ collinear to but distinct from $x$ are special to $p$, and all points of $\xi$ not collinear to $x$ are opposite $p$.  
\end{lemma}

\subsection{Automorphisms, opposition, and domesticity} \label{e77}The longest element $w_0$ of a spherical Coxeter group $W$ induces an automorphism $\sigma_0$ of the Coxeter graph (or Dynkin diagram) $\Pi$ called the \textit{opposition relation} on the type set. If $W$ is of type $\mathsf{A}_n$ with $n\geq 2$, of $\mathsf{D}_n$ with $n\geq 3$ odd, or $\mathsf{E}_6$ then $\sigma_0$ is the unique order~$2$ automorphism of $\Pi$, and in all other irreducible cases $\sigma_0$ is trivial. Chambers $A,B$ of a spherical building $(\Delta,\delta)$ of type $W$ are \textit{opposite} if $\delta(A,B)=w_0$. Simplices $\alpha$ and $\beta$ in a spherical building are \textit{opposite} if they have ``opposite types'' (that is, $\mathrm{type}(\beta)=(\mathrm{type}(\alpha))^{\sigma_0}$) and there exists a chamber $A$ containing $\alpha$ and a chamber $B$ containing $\beta$ such that $A$ and $B$ are opposite. Two objects in a Lie incidence geometry are opposite if the corresponding simplices of the building are opposite. 

It is convenient to call an automorphism (of a building, or Lie incidence geometry) an \textit{oppomorphism} if it maps each object to an object of the opposite type (that is, the companion automorphism is the opposition relation on the type set). For example, oppomorphisms of a $\mathsf{D}_n$ building are type preserving if $n$ is even, and interchange types $n-1$ and $n$ if $n$ is odd (using standard Bourbaki labelling). 

An automorphism of a spherical building is called \textit{anisotropic} if it maps all chambers to opposite chambers, and is called \emph{domestic} if it does not map any chamber to an opposite chamber. There exists an almost complete classification of domestic automorphisms of spherical buildings (see \cite{Lam-Mal:23,NPVV,PVMexc,PVMexc2,PVMclass,PVMexc4}), and we will use parts of this classification in the proof of our main theorem (Theorem~\ref{thm:A}). 

If an automorphism of a spherical building does not map any simplex of type $J$ to an opposite (with $J$ stable under the opposition relation), then we say that the automorphism is \emph{$J$-domestic}. If $J$ is the type of points, lines or symps (of a polar space or a long root subgroup geometry), then we also sometimes call the automorphism \emph{point-domestic}, \emph{line-domestic}, or \emph{symp-domestic}, respectively. In the case of polar spaces, we call a collineation $\ell$-domestic if no subspace of projective dimension $\ell$ is mapped onto an opposite (so point-domestic is the same thing as $0$-domestic; note that there is a shift in indexing of types here, with $0$-domestic being the same as $\{1\}$-domestic).

An automorphism of a spherical building is called \emph{capped} if, whenever there exist simplices of types $J$ and $J'$ mapped to respective opposite simplices, then there exists a simplex of type $J\cup J'$ mapped onto an opposite. In case the rank of the spherical building is at least 3, automorphisms that are not capped only exist when the building has rank 2 Fano plane residues (the unique projective plane with 3 points per line), see \cite{PVM:19b}. We record the following special case of \cite[Theorem~1]{PVM:19b}.

\begin{prop}\label{uncapped}
Let $\theta$ be a domestic automorphism of a building of type $\mathsf{E_6,E_7}$ or $\mathsf{E_8}$ which is not capped.  Let $i$ be the type of the polar node of the corresponding Dynkin diagram.  Then there exists a simplex of cotype $\{i\}$ which is mapped onto an opposite. 
\end{prop}

Recall that an automorphism $\theta$ of a building is \textit{uniclass} if its displacement spectrum is contained in a single $\sigma$-conjugacy class in the Coxeter group (where $\sigma$ is the companion automorphism of $\theta$). By \cite[Theorem~3.4]{NPV} every uniclass automorphism of a thick irreducible spherical building is capped. 

We now recall opposition and fix diagrams of automorphisms, from \cite{PVM:19a}.  Let $\Pi$ be the Coxeter diagram of an irreducible thick spherical building $\Delta$ and let $\sigma_0$ be the opposition relation on types. Let $\theta$ be an automorphism of $\Delta$, let $\sigma$ be the automorphism induced on $\Pi$ by $\theta$, let $\varphi$ be another automorphism of $\Pi$ and let $J$ be a set of \emph{distinguished} orbits of vertices of $\Pi$ under the action of $\<\varphi\>$. Then we say that $(\Pi,J,\varphi)$ is
\begin{compactenum}
\item the \emph{fix diagram} for $\theta$ if $\sigma=\varphi$ and $J$ consists precisely of the types of the minimal simplices of $\Delta$ fixed by $\theta$;
\item the \emph{opposition diagram} for $\theta$ if $\sigma=\varphi\sigma_0$ and $J$ consists precisely of the types of the minimal simplices of $\Delta$ sent to opposites b $\theta$.
\end{compactenum} 

We visualise fix and opposition diagrams by encircling the members of $J$ on the diagram~$\Pi$, making the action of $\varphi$ apparent by bending edges originating from a common vertex and belonging to the same orbit of $\varphi$.  This visualisation is strongly inspired by the \emph{indices} introduced by Jacques Tits in \cite{Tits:66} and has since become standard. 

In \cite{PVMexc,PVMclass} we introduced symbols to denote opposition diagrams (inspired by the indices from~\cite{Tits:66}), and we use the same symbols for fix diagrams. We list the symbols that are relevant for the present paper in~\cref{FixDiagrams}. For example, if $\theta$ is an automorphism of an $\mathsf{A}_7$ building with fix diagram $\mathsf{A}_{7;3}^2$ then $\theta$ is type preserving, and fixes vertices of types $2,4,6$ (and only these types), while if $\theta$ has opposition diagram $\mathsf{A}_{7;3}^2$ then $\theta$ induces the order $2$ automorphism on the type set, and maps vertices of types $2,4,6$ (and only these types) to opposites.

\renewcommand{\arraystretch}{1.6}
\begin{table}[t]
\begin{center}
\noindent\begin{tabular}{|c|c|c|c|}
\hline
Type& Symbol&Diagram&\\
\hline\hline
\multirow{2}{*}{$\mathsf{A}_{2n+1}$}
& $\mathsf{^2A}_{2n+1;n+1}^1$
&
\begin{tikzpicture}[scale=0.5,baseline=-0.5ex]
\node at (0,0.8) {};
\node at (0,-0.8) {};
\node [inner sep=0.8pt,outer sep=0.8pt] at (1,0.5) (1a) {$\bullet$};
\node [inner sep=0.8pt,outer sep=0.8pt] at (1,-0.5) (1b) {$\bullet$};
\node [inner sep=0.8pt,outer sep=0.8pt] at (2,0.5) (2a) {$\bullet$};
\node [inner sep=0.8pt,outer sep=0.8pt] at (2,-0.5) (2b) {$\bullet$};
\node [inner sep=0.8pt,outer sep=0.8pt]  at (3,0.5) (3a) {$\bullet$};
\node [inner sep=0.8pt,outer sep=0.8pt] at (3,-0.5) (3b) {$\bullet$};
\node [inner sep=0.8pt,outer sep=0.8pt] at (5,0.5) (4a) {$\bullet$};
\node [inner sep=0.8pt,outer sep=0.8pt] at (5,-0.5) (4b) {$\bullet$};
\node [inner sep=0.8pt,outer sep=0.8pt] at (0,0) (5) {$\bullet$};
\draw (0,0) to [bend left=45] (1,0.5);
\draw (0,0) to [bend right=45] (1,-0.5);
\draw (1,0.5)--(3,0.5);
\draw (1,-0.5)--(3,-0.5);
\draw [dashed] (3,0.5)--(5,0.5);
\draw [dashed] (3,-0.5)--(5,-0.5);
\draw [line width=0.5pt,line cap=round,rounded corners] (1a.north west)  rectangle (1b.south east);
\draw [line width=0.5pt,line cap=round,rounded corners] (2a.north west)  rectangle (2b.south east);
\draw [line width=0.5pt,line cap=round,rounded corners] (3a.north west)  rectangle (3b.south east);
\draw [line width=0.5pt,line cap=round,rounded corners] (4a.north west)  rectangle (4b.south east);
\draw [line width=0.5pt,line cap=round,rounded corners] (5.north west)  rectangle (5.south east);
\end{tikzpicture}& $n\geq 1$\\
& $\mathsf{A}_{2n+1;n}^2$
&
\begin{tikzpicture}[scale=0.5,baseline=-0.5ex]
\node at (0,0.3) {};
\node at (0,-0.3) {};
\node [inner sep=0.8pt,outer sep=0.8pt] at (-5,0) (-5) {$\bullet$};
\node [inner sep=0.8pt,outer sep=0.8pt] at (-4,0) (-4) {$\bullet$};
\node [inner sep=0.8pt,outer sep=0.8pt] at (-3,0) (-3) {$\bullet$};
\node [inner sep=0.8pt,outer sep=0.8pt] at (-2,0) (-2) {$\bullet$};
\node [inner sep=0.8pt,outer sep=0.8pt] at (-1,0) (-1) {$\bullet$};
\node [inner sep=0.8pt,outer sep=0.8pt] at (1,0) (1) {$\bullet$};
\node [inner sep=0.8pt,outer sep=0.8pt] at (2,0) (2) {$\bullet$};
\node [inner sep=0.8pt,outer sep=0.8pt] at (3,0) (3) {$\bullet$};
\node [inner sep=0.8pt,outer sep=0.8pt] at (4,0) (4) {$\bullet$};
\node [inner sep=0.8pt,outer sep=0.8pt] at (5,0) (5) {$\bullet$};
\draw (-5,0)--(-1,0);
\draw (1,0)--(5,0);
\draw [style=dashed] (-1,0)--(1,0);
\draw [line width=0.5pt,line cap=round,rounded corners] (-4.north west)  rectangle (-4.south east);
\draw [line width=0.5pt,line cap=round,rounded corners] (-2.north west)  rectangle (-2.south east);
\draw [line width=0.5pt,line cap=round,rounded corners] (2.north west)  rectangle (2.south east);
\draw [line width=0.5pt,line cap=round,rounded corners] (4.north west)  rectangle (4.south east);
\phantom{\draw [line width=0.5pt,line cap=round,rounded corners] (-5.north west)  rectangle (-5.south east);}
\end{tikzpicture} & $n\geq 1$
\\ \hline
& $\sD_{n;n-2i}^1$
& \begin{tikzpicture}[scale=0.5,baseline=-0.5ex]
\node at (0,0.8) {};
\node [inner sep=0.8pt,outer sep=0.8pt] at (-5,0) (-5) {$\bullet$};
\node [inner sep=0.8pt,outer sep=0.8pt] at (-4,0) (-4) {$\bullet$};
\node [inner sep=0.8pt,outer sep=0.8pt] at (-2,0) (-2) {$\bullet$};
\node [inner sep=0.8pt,outer sep=0.8pt] at (-1,0) (-1) {$\bullet$};
\node [inner sep=0.8pt,outer sep=0.8pt] at (1,0) (1) {$\bullet$};
\node [inner sep=0.8pt,outer sep=0.8pt] at (2,0) (2) {$\bullet$};
\node [inner sep=0.8pt,outer sep=0.8pt] at (3,0.5) (5a) {$\bullet$};
\node [inner sep=0.8pt,outer sep=0.8pt] at (3,-0.5) (5b) {$\bullet$};
\draw (-5,0)--(-4,0);
\draw (-2,0)--(-1,0);
\draw (1,0)--(2,0);
\draw (2,0) to (3,0.5);
\draw (2,0) to   (3,-0.5);
\draw [style=dashed] (-4,0)--(-2,0);
\draw [style=dashed] (-1,0)--(1,0);
\draw [line width=0.5pt,line cap=round,rounded corners] (-5.north west)  rectangle (-5.south east);
\draw [line width=0.5pt,line cap=round,rounded corners] (-4.north west)  rectangle (-4.south east);
\draw [line width=0.5pt,line cap=round,rounded corners] (-2.north west)  rectangle (-2.south east);
\node [below] at (-2,-0.25) {\footnotesize $n\!\!-\!\!2i$};
\end{tikzpicture} & $n\geq 3$, $2\leq 2i\leq n-1$\\
$\mathsf{D}_n$ 
& $\sD_{n;n-2i+1}^1$
& \begin{tikzpicture}[scale=0.5,baseline=-0.5ex]
\node at (0,0.8) {};
\node [inner sep=0.8pt,outer sep=0.8pt] at (-5,0) (-5) {$\bullet$};
\node [inner sep=0.8pt,outer sep=0.8pt] at (-4,0) (-4) {$\bullet$};
\node [inner sep=0.8pt,outer sep=0.8pt] at (-2,0) (-2) {$\bullet$};
\node [inner sep=0.8pt,outer sep=0.8pt] at (-1,0) (-1) {$\bullet$};
\node [inner sep=0.8pt,outer sep=0.8pt] at (1,0) (1) {$\bullet$};
\node [inner sep=0.8pt,outer sep=0.8pt] at (2,0) (2) {$\bullet$};
\node [inner sep=0.8pt,outer sep=0.8pt] at (3,0.5) (5a) {$\bullet$};
\node [inner sep=0.8pt,outer sep=0.8pt] at (3,-0.5) (5b) {$\bullet$};
\draw (-5,0)--(-4,0);
\draw (-2,0)--(-1,0);
\draw (2,0) to [bend left] (3,0.5);
\draw (1,0)--(2,0);
\draw (2,0) to [bend right=45]  (3,-0.5);
\draw [style=dashed] (-4,0)--(-2,0);
\draw [style=dashed] (-1,0)--(1,0);
\draw [line width=0.5pt,line cap=round,rounded corners] (-5.north west)  rectangle (-5.south east);
\draw [line width=0.5pt,line cap=round,rounded corners] (-4.north west)  rectangle (-4.south east);
\draw [line width=0.5pt,line cap=round,rounded corners] (-2.north west)  rectangle (-2.south east);
\node [below] at (-2,-0.25) {\footnotesize $n\!\!-\!\!2i\!\!+\!\!1$};
\end{tikzpicture}& $n\geq 4$, $4\leq 2i\leq n$
\\
& $\sD_{n;n-1}^1$
& \begin{tikzpicture}[scale=0.5,baseline=-0.5ex]
\node at (0,0.8) {};
\node at (0,-0.8) {};
\node [inner sep=0.8pt,outer sep=0.8pt] at (-4,0) (-4) {$\bullet$};
\node [inner sep=0.8pt,outer sep=0.8pt] at (-3,0) (-3) {$\bullet$};
\node [inner sep=0.8pt,outer sep=0.8pt] at (-2,0) (-2) {$\bullet$};
\node [inner sep=0.8pt,outer sep=0.8pt] at (-1,0) (-1) {$\bullet$};
\node [inner sep=0.8pt,outer sep=0.8pt] at (1,0) (1) {$\bullet$};
\node [inner sep=0.8pt,outer sep=0.8pt] at (2,0) (2) {$\bullet$};
\node [inner sep=0.8pt,outer sep=0.8pt] at (3,0) (3) {$\bullet$};
\node [inner sep=0.8pt,outer sep=0.8pt] at (4,0.5) (5a) {$\bullet$};
\node [inner sep=0.8pt,outer sep=0.8pt] at (4,-0.5) (5b) {$\bullet$};
\draw (-4,0)--(-1,0);
\draw (1,0)--(3,0);
\draw (3,0) to [bend left] (4,0.5);
\draw (3,0) to [bend right=45] (4,-0.5);
\draw [style=dashed] (-1,0)--(1,0);
\draw [line width=0.5pt,line cap=round,rounded corners] (-4.north west)  rectangle (-4.south east);
\draw [line width=0.5pt,line cap=round,rounded corners] (-2.north west)  rectangle (-2.south east);
\draw [line width=0.5pt,line cap=round,rounded corners] (2.north west)  rectangle (2.south east);
\draw [line width=0.5pt,line cap=round,rounded corners] (-3.north west)  rectangle (-3.south east);
\draw [line width=0.5pt,line cap=round,rounded corners] (-1.north west)  rectangle (-1.south east);
\draw [line width=0.5pt,line cap=round,rounded corners] (1.north west)  rectangle (1.south east);
\draw [line width=0.5pt,line cap=round,rounded corners] (3.north west)  rectangle (3.south east);
\draw [line width=0.5pt,line cap=round,rounded corners] (5a.north west)  rectangle (5b.south east);
\end{tikzpicture} & $n\geq 3$
\\ 
$\sD_{2n}$
& $\sD_{2n;n}^2$
& \begin{tikzpicture}[scale=0.5,baseline=-0.5ex]
\node at (0,-0.8) {};
\node [inner sep=0.8pt,outer sep=0.8pt] at (-5,0) (-5) {$\bullet$};
\node [inner sep=0.8pt,outer sep=0.8pt] at (-4,0) (-4) {$\bullet$};
\node [inner sep=0.8pt,outer sep=0.8pt] at (-3,0) (-3) {$\bullet$};
\node [inner sep=0.8pt,outer sep=0.8pt] at (-2,0) (-2) {$\bullet$};
\node [inner sep=0.8pt,outer sep=0.8pt] at (-1,0) (-1) {$\bullet$};
\node [inner sep=0.8pt,outer sep=0.8pt] at (1,0) (1) {$\bullet$};
\node [inner sep=0.8pt,outer sep=0.8pt] at (2,0) (2) {$\bullet$};
\node [inner sep=0.8pt,outer sep=0.8pt] at (3,0) (3) {$\bullet$};
\node [inner sep=0.8pt,outer sep=0.8pt] at (4,0) (4) {$\bullet$};
\node [inner sep=0.8pt,outer sep=0.8pt] at (5,0.5) (5a) {$\bullet$};
\node [inner sep=0.8pt,outer sep=0.8pt] at (5,-0.5) (5b) {$\bullet$};
\draw (-5,0)--(-1,0);
\draw (1,0)--(4,0);
\draw (4,0) to   (5,0.5);
\draw (4,0) to   (5,-0.5);
\draw [style=dashed] (-1,0)--(1,0);
\draw [line width=0.5pt,line cap=round,rounded corners] (-4.north west)  rectangle (-4.south east);
\draw [line width=0.5pt,line cap=round,rounded corners] (-2.north west)  rectangle (-2.south east);
\draw [line width=0.5pt,line cap=round,rounded corners] (2.north west)  rectangle (2.south east);
\draw [line width=0.5pt,line cap=round,rounded corners] (4.north west)  rectangle (4.south east);
\draw [line width=0.5pt,line cap=round,rounded corners] (5b.north west)  rectangle (5b.south east);
\phantom{\draw [line width=0.5pt,line cap=round,rounded corners] (-5.north west)  rectangle (-5.south east);}
\end{tikzpicture} & $n\geq 2$
\\ \hline\hline
\multirow{2}{*}{$\mathsf{E_6}$} & $\mathsf{^2E_{6;4}}$ &
\begin{tikzpicture}[scale=0.5,baseline=-0.5ex]
\node at (0,0.8) {};
\node [inner sep=0.8pt,outer sep=0.8pt] at (-2,0) (2) {$\bullet$};
\node [inner sep=0.8pt,outer sep=0.8pt] at (-1,0) (4) {$\bullet$};
\node [inner sep=0.8pt,outer sep=0.8pt] at (0,-0.5) (5) {$\bullet$};
\node [inner sep=0.8pt,outer sep=0.8pt] at (0,0.5) (3) {$\bullet$};
\node [inner sep=0.8pt,outer sep=0.8pt] at (1,-0.5) (6) {$\bullet$};
\node [inner sep=0.8pt,outer sep=0.8pt] at (1,0.5) (1) {$\bullet$};
\draw (-2,0)--(-1,0);
\draw (-1,0) to [bend left=45] (0,0.5);
\draw (-1,0) to [bend right=45] (0,-0.5);
\draw (0,0.5)--(1,0.5);
\draw (0,-0.5)--(1,-0.5);
\draw [line width=0.5pt,line cap=round,rounded corners] (2.north west)  rectangle (2.south east);
\draw [line width=0.5pt,line cap=round,rounded corners] (4.north west)  rectangle (4.south east);
\draw [line width=0.5pt,line cap=round,rounded corners] (3.north west)  rectangle (5.south east);
\draw [line width=0.5pt,line cap=round,rounded corners] (1.north west)  rectangle (6.south east);
\end{tikzpicture} &
\\
& $\mathsf{E_{6;2}}$ &
\begin{tikzpicture}[scale=0.5]
\node at (0,0.3) {};
\node [inner sep=0.8pt,outer sep=0.8pt] at (-2,0) (1) {$\bullet$};
\node [inner sep=0.8pt,outer sep=0.8pt] at (-1,0) (3) {$\bullet$};
\node [inner sep=0.8pt,outer sep=0.8pt] at (0,0) (4) {$\bullet$};
\node [inner sep=0.8pt,outer sep=0.8pt] at (1,0) (5) {$\bullet$};
\node [inner sep=0.8pt,outer sep=0.8pt] at (2,0) (6) {$\bullet$};
\node [inner sep=0.8pt,outer sep=0.8pt] at (0,1) (2) {$\bullet$};
\draw (-2,0)--(2,0);
\draw (0,0)--(0,1);
\draw [line width=0.5pt,line cap=round,rounded corners] (1.north west)  rectangle (1.south east);
\draw [line width=0.5pt,line cap=round,rounded corners] (6.north west)  rectangle (6.south east);
\end{tikzpicture}&\\ \hline
\multirow{2}{*}{$\mathsf{E_7}$}
& $\mathsf{E_{7;3}}$
&\begin{tikzpicture}[scale=0.5]
\node at (0,0.3) {};
\node [inner sep=0.8pt,outer sep=0.8pt] at (-2,0) (1) {$\bullet$};
\node [inner sep=0.8pt,outer sep=0.8pt] at (-1,0) (3) {$\bullet$};
\node [inner sep=0.8pt,outer sep=0.8pt] at (0,0) (4) {$\bullet$};
\node [inner sep=0.8pt,outer sep=0.8pt] at (1,0) (5) {$\bullet$};
\node [inner sep=0.8pt,outer sep=0.8pt] at (2,0) (6) {$\bullet$};
\node [inner sep=0.8pt,outer sep=0.8pt] at (3,0) (7) {$\bullet$};
\node [inner sep=0.8pt,outer sep=0.8pt] at (0,1) (2) {$\bullet$};
\draw (-2,0)--(3,0);
\draw (0,0)--(0,1);
\draw [line width=0.5pt,line cap=round,rounded corners] (1.north west)  rectangle (1.south east);
\draw [line width=0.5pt,line cap=round,rounded corners] (6.north west)  rectangle (6.south east);
\draw [line width=0.5pt,line cap=round,rounded corners] (7.north west)  rectangle (7.south east);
\end{tikzpicture} &
\\ 
& $\mathsf{E_{7;4}}$ &
\begin{tikzpicture}[scale=0.5]
\node [inner sep=0.8pt,outer sep=0.8pt] at (-2,0) (1) {$\bullet$};
\node [inner sep=0.8pt,outer sep=0.8pt] at (-1,0) (3) {$\bullet$};
\node [inner sep=0.8pt,outer sep=0.8pt] at (0,0) (4) {$\bullet$};
\node [inner sep=0.8pt,outer sep=0.8pt] at (1,0) (5) {$\bullet$};
\node [inner sep=0.8pt,outer sep=0.8pt] at (2,0) (6) {$\bullet$};
\node [inner sep=0.8pt,outer sep=0.8pt] at (3,0) (7) {$\bullet$};
\node [inner sep=0.8pt,outer sep=0.8pt] at (0,1) (2) {$\bullet$};
\draw (-2,0)--(3,0);
\draw (0,0)--(0,1);
\draw [line width=0.5pt,line cap=round,rounded corners] (1.north west)  rectangle (1.south east);
\draw [line width=0.5pt,line cap=round,rounded corners] (3.north west)  rectangle (3.south east);
\draw [line width=0.5pt,line cap=round,rounded corners] (4.north west)  rectangle (4.south east);
\draw [line width=0.5pt,line cap=round,rounded corners] (6.north west)  rectangle (6.south east);
\end{tikzpicture}&\\
\hline
$\mathsf{E_8}$
&$\mathsf{E_{8;4}}$
&
\begin{tikzpicture}[scale=0.5]
\node at (0,0.3) {};
\node [inner sep=0.8pt,outer sep=0.8pt] at (-2,0) (1) {$\bullet$};
\node [inner sep=0.8pt,outer sep=0.8pt] at (-1,0) (3) {$\bullet$};
\node [inner sep=0.8pt,outer sep=0.8pt] at (0,0) (4) {$\bullet$};
\node [inner sep=0.8pt,outer sep=0.8pt] at (1,0) (5) {$\bullet$};
\node [inner sep=0.8pt,outer sep=0.8pt] at (2,0) (6) {$\bullet$};
\node [inner sep=0.8pt,outer sep=0.8pt] at (3,0) (7) {$\bullet$};
\node [inner sep=0.8pt,outer sep=0.8pt] at (4,0) (8) {$\bullet$};
\node [inner sep=0.8pt,outer sep=0.8pt] at (0,1) (2) {$\bullet$};
\draw (-2,0)--(4,0);
\draw (0,0)--(0,1);
\draw [line width=0.5pt,line cap=round,rounded corners] (1.north west)  rectangle (1.south east);
\draw [line width=0.5pt,line cap=round,rounded corners] (6.north west)  rectangle (6.south east);
\draw [line width=0.5pt,line cap=round,rounded corners] (7.north west)  rectangle (7.south east);
\draw [line width=0.5pt,line cap=round,rounded corners] (8.north west)  rectangle (8.south east);
\end{tikzpicture}&\\
\hline\end{tabular}
\caption{Some fix and opposition diagrams and their symbols}\label{FixDiagrams}
\end{center}
\end{table}

It is shown in \cite[\S1.2]{NPV} that there is a duality between the opposition and fix diagrams of uniclass automorphisms in the sense that any two uniclass automorphisms with the same fixed (respectively opposition) diagram~$\mathsf{X}$ will necessarily have the same opposition (respectively fix) diagram~$\mathsf{Y}$, and this correspondence is involutory. We list this correspondence in \cref{Duality} for the diagrams that are relevant for this paper. 

\begin{table}
\begin{center}
\begin{tabular}{|rcl|rcl|}\hline
 $\mathsf{^2A}_{2n+1;n+1}^1$& $\leftrightarrow$ & $\mathsf{A}_{2n+1;n}^2$ & $\mathsf{D}^1_{n;i}$&$\leftrightarrow$ & $\mathsf{D}^1_{n;n-i}$ \\ \hline $\mathsf{D}_{2n;n}^2$ & $\leftrightarrow$ & $\mathsf{D}^2_{2n;n}$ & $\mathsf{^2E_{6;4}}$ & $\leftrightarrow$ & $\mathsf{E_{6;2}}$\\ \hline$\mathsf{E_{7;3}}$ & $\leftrightarrow$ & $\mathsf{E_{7;4}}$ & $\mathsf{E_{8;4}}$ & $\leftrightarrow$ & $\mathsf{E_{8;4}}$\\ \hline\end{tabular}
 \caption{Fix and opposition diagram duality for uniclass automorphisms}\label{Duality}
\end{center}
\end{table}

Recall the definition of a $D$-kangaroo of a long root subgroup geometry from the introduction (where $D\subseteq \{0,1,2,2',3\}$). If $D=\{d\}$ is a singleton we write $d$-kangaroo in place of $\{d\}$-kangaroo. 

On occasion it will be more convenient to argue in a Lie incidence geometry other than the long root subgroup geometry. In such a case, an automorphism of the long root subgroup geometry with spectral gaps (for example, a $\{1,2'\}$-kangaroo) may not have a spectral gap property when viewed as an automorphism of another Lie incidence geometry. It is convenient in such an instance to call a $D$-kangaroo (with $D\subseteq\{0,1,2,2',3\}$) of the long root geometry, when viewed as an automorphism of another Lie incidence geometry associated with the same building, a \textit{polar $D$-kangaroo} (the word ``polar'' is used here to indicate the the kangaroo property is only with respect to the polar node geometry -- that is, the long root subgroup geometry). 
%
%

An example of the phenomenon noted above is that in type $\mathsf{E}_7$, where the long root geometry is $\Gamma=\mathsf{E}_{7,1}(\K)$, it is sometimes more convenient to work in the geometry $\Gamma'=\mathsf{E}_{7,7}(\K)$. The points of $\Gamma$ are then the symps of $\Gamma'$. Collinear points of $\Gamma$ correspond to symps of $\Gamma'$ intersecting in a maximal singular subspace (and, in general, we will call such symps \emph{adjacent}); a special pair of  points corresponds to a pair of disjoint symps for which there exists a unique symp intersecting both in respective maximal singular subspaces.

\subsection{Characterisation of uniclass automorphisms from~\cite{NPV}}

 In \cite{NPV} we characterised uniclass automorphisms of thick irreducible spherical buildings in terms of their fixed structure. The following theorem is a consequence of this characterisation for the simply laced case.

\begin{thm}[\cite{NPV}]\label{thm:basictheorem}
Let $\theta$ be a non-trivial automorphism of a thick irreducible spherical building $\Delta$ of rank at least~$2$ and with simply laced Coxeter diagram. Then $\theta$ is uniclass if and only if $\theta$ is either anisotropic, or:
\begin{compactenum}[$(1)$]
\item $\Delta$ has type $\sA_{2n+1}$ $(n\geq 1)$ and in the associated projective space
\begin{compactenum}[$-$]
\item $\theta$ is a fix point free collineation fixing a line spread elementwise, or
\item $\theta$ is a symplectic polarity (a polarity fixing a symplectic polar space of rank~$n$). 
\end{compactenum}
\item $\Delta$ has type $\sD_n$ ($n\geq 4$) and in the associated polar space
\begin{compactenum}[$-$]
\item $\theta$ is a collineation whose fixed points form an ideal subspace, or
\item $\theta$ is a fix point free collineation fixing a line spread elementwise.
\end{compactenum}
\item $\Delta=\sE_6(\K)$ (with $\K$ a field) and 
\begin{compactenum}[$-$]
\item $\theta$ is a symplectic polarity (a polarity fixing a standard split metasymplectic space), or
\item $\theta$ is a collineation fixing an ideal Veronesian pointwise in $\sE_{6,1}(\KK)$. 
\end{compactenum}
\item $\Delta=\sE_7(\K)$ (with $\K$ a field) and
\begin{compactenum}[$-$]
\item the fixed point structure of $\theta$ in $\sE_{7,1}(\K)$ is a fully embedded metasymplectic space $\sF_4(\K,\LL)$ with $\LL$ a quadratic extension of $\K$, isometrically embedded as a long root subgroup geometry, or
\item the fixed point structure of $\theta$ in $\sE_{7,7}(\K)$ is an ideal dual polar Veronesian.
\end{compactenum}
\item $\Delta=\sE_8(\K)$ (with $\K$ a field) and the fixed point structure of $\theta$ in $\sE_{8,8}(\K)$ is a fully (and automatically isometrically) embedded metasymplectic space $\sF_4(\K,\mathbb{H})$ with $\mathbb{H}$ either a quaternion algebra over $\K$ or an inseparable quadratic field extension of degree~$4$ in characteristic~$2$.
\end{compactenum}
\end{thm}

The definition of the various fixed structures in the theorem (line spreads, ideal subspace, standard split metasymplectic space, ideal Veronesian, and so on), as far as they are essential to understand our proofs, will be given in the relevant subsections of Section~\ref{sec:proofs}. We call these fixed structures \textit{Weyl substructures} in $\Delta$. These Weyl substructures are large and highly structured subsets of the (simplicial) building. Indeed, each Weyl substructure $\Delta'$ is itself a thick spherical building. 

Putting the above theorem in tabular form, we obtain the list of Weyl substructures for simply laced diagrams, along with their Coxeter type, given in Table~\ref{AllWS}. In the table, the \textit{absolute type} is the Coxeter type of the ambient building $\Delta$, and the \textit{relative type} is the Coxeter type of the Weyl substructure~$\Delta'$. Moreover, in each case we list the fix diagram of the associated automorphisms of $\Delta$ fixing $\Delta'$ (the opposition diagram can then be obtained using the involution in Table~\ref{Duality}).

\begin{table}[h!]
\renewcommand{\arraystretch}{1.5}
\begin{tabular}{c|c||c||c}
\mbox{Abs. type} & \mbox{Rel. type} & \mbox{Description of Weyl substructure} & \mbox{Fix diagram}\\ \hline\hline
 \multirow{2}{*}{$\mathsf{A}_{2n-1}$} & $\mathsf{B}_{n}$ & symplectic polar space of rank~$n$ & ${^2}\sA_{2n-1;n}^1$ \\
 &   $\mathsf{A}_{n-1}$  & composition line spread  & ${^1}\sA_{2n-1;n-1}^2$  \\ \hline 
 \multirow{2}{*}{$\mathsf{D}_{n}$} & $\mathsf{B}_{i}$ & ideal subspace of rank $i$ & $\sD_{n;i}^1$   \\
 &   $\mathsf{B}_{n/2}$  & composition line spread & $\sD_{n;n/2}^2$  \\ \hline 
  \multirow{2}{*}{$\mathsf{E}_{6}$} & $\mathsf{F}_{4}$& standard split metasymplectic space & ${^2}\sE_{6;4}$   \\ 
    & $\mathsf{A}_{2}$  & ideal Veronesian  & $\sE_{6;2}$  \\ \hline
     \multirow{2}{*}{$\mathsf{E}_{7}$} & $\mathsf{F}_{4}$& partial composition spread  & $\sE_{7;4}$  \\ 
    & $\mathsf{B}_{3}$  & ideal dual polar quaternion Veronesian  & $\sE_{7;3}$  \\ \hline
         \multirow{1}{*}{$\mathsf{E}_{8}$} & $\mathsf{F}_{4}$& quaternion metasymplectic space& $\sE_{8;4}$   \\ 
 \end{tabular}
\vspace{6pt}
\caption{Weyl substructures\label{AllWS}}
\end{table}

\section{Proof of the Main Result---Classical cases}\label{sec:proofs}

In this section we prove Theorem~\ref{thm:A} for classical (simply laced) types.

\subsection{Buildings of type $\sA_n$ (projective spaces)}
A \emph{symplectic polarity} 
of a projective space is a polarity such that every point 
is contained in its image. Symplectic polarities are always related to a 
nondegenerate alternating form in the underlying vector space, and hence only exist 
for projective spaces of odd rank over commutative fields (see \cite{TTVM3}). A 
\emph{line spread} of a projective space is a partition of the point set into lines.  A 
line spread is a \emph{composition spread} if it induces a line spread in every 
subspace spanned by members of the spread. Note that this is automatic if a line 
spread is elementwise fixed under a collineation.

The following theorem proves Theorem~\ref{thm:A} for projective spaces. 

\begin{thm}
An non-trivial automorphism of $\mathsf{A}_{n,1}(\K)=\PG(n,\K)$ is a polar $\{1,2'\}$-kangaroo if, and only if, either its is anisotropic, or it is a symplectic polarity (and then it is a polar $\{1,2',3\}$-kangaroo), or it elementwise fixes a line spread (and then it is a polar $\{0,1,2'\}$-kangaroo). Consequently, $\theta$ is a polar $\{1,2'\}$-kangaroo if, and only if, it is uniclass.
\end{thm}

\begin{proof} Suppose that $\theta$ is a polar $\{1,2'\}$-kangaroo of $\mathsf{A}_{n,1}(\K)=\PG(n,\K)$, for some skew field $\K$. Suppose first that $\theta$ is type preserving and non-trivial. The points of the long root subgroup geometry are the incident point-hyperplane pairs of $\PG(n,\K)$, collinear if sharing either the point or the hyperplane. In order to avoid confusion with collinearity in the projective space, we call collinear points in the long root subgroup geometry \emph{adjacent}. It is easy to deduce the following remaining mutual positions:

\begin{compactenum}
\item[$\bullet$] The two point-hyperplane pairs $(p,H)$ and $(p',H')$ correspond to symplectic points (distance $2$) in the long root subgroup geometry if and only if $p\in H'$ and $p'\in H$.
\item[$\bullet$] The two point-hyperplane pairs $(p,H)$ and $(p',H')$ correspond to special points (distance $2'$) if and only if either $p\in H'$ and $p'\notin H$, or $p\notin H'$ and $p'\in H$. 
\item[$\bullet$] The two point-hyperplane pairs $(p,H)$ and $(p',H')$ correspond to opposite points (distance $3$) if and only if $p\notin H'$ and $p'\notin H$. 
\end{compactenum}

Observe that if $\theta$ fixes a hyperplane, then it fixes all points of that hyperplane (because $\theta$ is a (polar) $1$-kangaroo). It follows that if $\theta$ fixes a hyperplane, then it is the identity (because the dual of the above observation implies that $\theta$ fixes all hyperplanes). Since we assumed that $\theta\neq\id$, no hyperplane is fixed. Dually, $\theta$ does not fix any point. 

Next observe that if $\theta$ maps a point $p$ to a distinct point $p^\theta$, then $\theta$ stabilises the line $pp^\theta$. To see this, suppose, for a contradiction, that $p^{\theta^2}$ does not belong to $pp^\theta$. Then we select a hyperplane $H$ containing $pp^\theta$ but not $p^{\theta^2}$. Then obviously $(p^\theta,H)$ and $(p^{\theta^2},H^\theta)$ are special, contradicting the fact that $\theta$ is a polar $2'$-kangaroo. 

Now by \cite[Proposition~3.3(i)]{PVMclass} the set of fixed lines of $\theta$ is a spread, which is necessarily a composition spread (as noted above). It follows from Theorem~\ref{thm:basictheorem} that $\theta$ is uniclass. 

Conversely, let $\theta$ fix a composition line spread in $\PG(n,\K)$ and let $(p,H)$ be a point-hyperplane pair. Let $L$ be the unique spread line containing $p$. There are two possibilities. Either $L\subseteq H$, and then, since $L^\theta=L$, we have $p^\theta\in H$ and $p\in H^\theta$, hence $(p,H)$ and $(p,H)^\theta$ are symplectic, or $L\cap H=\{p\}$ and then $p\notin H^\theta$ and $p^\theta\notin H$, so that $(p,H)$ and $(p,H)^\theta$ are opposite. Hence $(p,H)$ and $(p,H)^\theta$ are never identical, adjacent or special. This proves that $\theta$ is a $\{0,1,2'\}$-kangaroo.   

Now assume that $\theta$ is a type interchanging polar $\{1,2'\}$-kangaroo. We may assume it is not anisotropic. Observe the following: If a point-hyperplane flag is mapped onto a symplectic one in the long root subgroup geometry, then there is a fixed point-hyperplane pair. To see this, suppose $(p,H)$ is symplectic to $(H^\theta,p^\theta)$. Then $p\in p^\theta$ and $(p,p^\theta)$ is mapped onto $(p^{\theta^2},p^\theta)$, which is collinear to $(p,p^\theta)$ if $p\neq p^{\theta^2}$. The latter is never the case since $\theta$ is a polar $1$-kangaroo. Hence $(p,p^\theta)$ is fixed. 

Moreover, note that if a point $p$ is contained in its image $p^\theta$, then $p=p^{\theta^2}$. For if not, then $(p,p^\theta)$ is distinct from but adjacent to $(p^{\theta^2},p^\theta)$, contradicting $\theta$ being a $1$-kangaroo. 

Since we may assume that $\theta$ is not anisotropic, the above observation implies that we may assume that $\theta$ fixes some point-hyperplane flag $(p_0,p_0^\theta)$ (and so $(p_0^\theta)^\theta=p_0$). Now note that for a point $p\in p_0^\theta$, the image $p^\theta$ contains  $p$. Indeed, if $p$ were not contained in $p^\theta$, then the point-hyperplane pairs $(p,p_0^\theta)$ and $(p_0,p^\theta)$ would be special, a contradiction (because the second is the image of the first under $\theta$). Hence every point in $p_0^\theta$ is incident with its image. Now dually, every hyperplane incident with some point of $p_0^\theta$ is incident with its image. Hence every point is incident with its image, and we conclude that $\theta$ is a symplectic polarity. Hence $\theta$ is uniclass by Theorem~\ref{thm:basictheorem}. 

Conversely, let $\theta$ be a symplectic polarity of $\PG(n,\K)$. Let $(p,H)$ be any point-hyperplane pair. Then there are exactly two possibilities. Either $H=p^\theta$, and then the pair is fixed, or $H\neq p^\theta$ and then, since $p\in p^\theta$ and $H^\theta\in H$, the pairs $(p,H)$ and $(H^\theta,p^\theta)=(p,H)^\theta$ are symplectic. Hence they are never adjacent and never special (they are never opposite either). Hence $\theta$ is a $\{1,2',3\}$-kangaroo. This completes the proof of the theorem. 
\end{proof}

\subsection{Buildings of types $\sD_n$ (Oriflamme complexes of polar spaces)}

In this section we consider collineations of hyperbolic polar spaces. Note that this includes the case of automorphisms of type $\sD_n$ buildings interchanging types $n-1$ and $n$. Trialities of $\sD_4$ will also be briefly considered (and eliminated). 

We begin with some preliminaries. 
Let $\Gamma$ be a polar space. An \emph{ovoid} is a set of points intersecting every maximal singular subspace in exactly one point. A subspace of $\Gamma$ is called \emph{ideal} if it induces an ovoid in the upper residue of each of its submaximal singular subspaces.

In order to smoothly transfer from the terminology of \cite{NPVV} to the notation and setting in the present paper, we note that, since the polar node of a $\mathsf{D}_n$ diagram is the node labelled~2 (in Bourbaki labelling), the points of the corresponding long root subgroup geometry are the lines of the polar space $\Gamma$, and we review the different mutual positions. Lines $L$ and $M$ of $\Gamma$ are:
\begin{compactenum}
\item[$\bullet$] adjacent (distance $1$) if they are contained in a common plane of $\Gamma$;
\item[$\bullet$] symplectic (distance $2$) if they are not adjacent and either they share a point, or they are contained in a common singular subspace;
\item[$\bullet$] special (distance $2'$) if there is a unique point on either of them collinear to all points of the other line;
\item[$\bullet$] opposite (distance $3$) if each point of either is collinear to exactly one point of the other line.
\end{compactenum}

We recall the following proposition from~\cite{NPVV}. 

\begin{prop}[{\cite[Proposition~3.8]{NPVV}}]\label{prop:NPVV1}
Let $\theta$ be a collineation of a hyperbolic polar space of rank at least~$3$. The following are equivalent:
\begin{compactenum}[$(i)$]
\item $\theta$ is non-trivial, not anisotropic, but maps no point to a distinct collinear point;
\item $\theta$ is a non-trivial polar $\{1,2'\}$-kangaroo fixing at least one point;
\item the fixed point set of $\theta$ is a nonempty ideal subspace. 
\end{compactenum}
\end{prop}


The following theorem, along with Proposition~\ref{prop:trialities} below, proves Theorem~\ref{thm:A} for hyperbolic polar spaces (buildings of type~$\mathsf{D}_n$). Note that in Case $(ii)$ we do not have to require that $\theta$ does not fix any point (if it does, then it is automatically the identity, which is also uniclass).

\begin{thm}\label{Bn1uni}
Let $\theta$ be a collineation of a hyperbolic polar space. Then $\theta$ is a polar $\{1,2'\}$-kangaroo if, and only if, $\theta$ is either anisotropic, or
\begin{compactenum}[$(i)$]
\item the fixed points of $\theta$ form an ideal subspace, or
\item it fixes element-wise a line spread.
\end{compactenum}
Consequently, $\theta$ is a polar $\{1,2'\}$-kangaroo if, and only if, it is uniclass.
\end{thm}

\begin{proof}
Let $\theta$ be a polar $\{1,2'\}$-kangaroo, and suppose that $\theta$ is not anisotropic and fixes some point. Then by Proposition~\ref{prop:NPVV1} the fix structure is a nonempty ideal subspace. So we may assume that $\theta$ is a collineation without fixed points. We must show that $\theta$ fixes a line spread element-wise. 

We first claim that there is at least one fixed line. Indeed, suppose not. Then, since $\theta$ is not anisotropic, there is a symplectic pair $(L,L^\theta)$ of lines. Suppose first that $L^\theta$ is collinear to $L$. Pick a point $p\in L$. If $p\perp p^{\theta^2}$, then $pp^\theta$ and $p^\theta p^{\theta^2}$ are either adjacent or equal, both contradictions to our hypotheses. Hence $p$ is not collinear to $p^{\theta^2}$. It follows that there is a unique point of $L$ collinear to $p^{\theta^2}$. Consequently, we may pick a point $q\in L$ distinct from $p$ and not collinear to $p^{\theta^2}$. The lines $qp^\theta$ and $q^\theta p^{\theta^2}=(qp^\theta)^\theta$ contain a pair of non-collinear points, hence they are either special or opposite. But by our assumption on $\theta$, they are opposite. However, $p^\theta$ is collinear to both $q^\theta$ and $p^{\theta^2}$, a contradiction, leading to a fixed line in this case. 

Now assume $L$ and $L^\theta$ intersect but are not contained in a plane. Let $p=L\cap L^\theta$. Since the polar space has rank at least~3, we can pick a plane $\pi$ on $L^\theta$ and a line $M\subseteq\pi$ with $p\in M\neq L^\theta$. The line $M^\theta$, which does not belong to $\pi$ since $\theta$ is a polar $1$-kangaroo, intersects $\pi$ in $p^\theta\in L^\theta\setminus\{p\}$ and so, it is disjoint from $M$. However, since $p^\theta$ is collinear to $M$, it is not opposite, hence must be symplectic and disjoint, bringing us back to the previous case, and hence again leading to a fixed line. 

So we may assume that $\theta$ fixes a line $L$. Pick $q\in L$ and note that $q\neq q^\theta$. Let $p$ be any point collinear with $L$. Then $p\perp p^\theta$ as otherwise $pq$ and $p^\theta q^\theta$ are special. Suppose $pp^\theta$ is not fixed. Then, as $\theta$ is in particular a polar $1$-kangaroo, $p$ is not collinear to $p^{\theta^2}$. This yields a point $r\in pp^\theta$ not collinear to $r^\theta$. Then $qr$ and $q^\theta r^\theta$ are special, again a contradiction. We conclude that $pp^\theta$ is fixed. 

Obviously $L\cap pp^\theta=\varnothing$, so $L,p,p^\theta$ generate a singular $3$-space $\Sigma$. Now let $x$ be an arbitrary point of $\Gamma$ not collinear to $L$. Then it is collinear to all points of a plane $\pi\subseteq\Sigma$. Since $\Sigma^\theta=\Sigma$, and since $\pi\neq\pi^\theta$ (otherwise each of its lines are fixed since they cannot be mapped onto adjacent ones) the intersection $\pi\cap\pi^\theta$ is a line $M$. The image $M^\theta$ is coplanar with $M$ (as both lie in $\pi^\theta$), which implies that $M=M^\theta$. So $x$ is collinear to the fixed line $M$ and, switching the roles of $L$ and $M$, we again conclude that $x\perp x^\theta$ and $xx^\theta$ is fixed. So the fixed lines form a spread, completing the ``only if'' direction of the proof.

Now assume that $\theta$ is a collineation that is not anisotropic. If the fixed points of $\theta$ form an ideal subspace, then by Proposition~\ref{prop:NPVV1} $\theta$ is a polar $\{1,2'\}$-kangaroo. So suppose now that $\theta$ fixes a line spread. Let $L$ be an arbitrary line. If $L$ is not fixed, then each point of $L$ is sent to a collinear point outside $L$. It follows that $L$ and $L^\theta$ are disjoint. Suppose that they are special. Then some point $x\in L$ is collinear to all points of $L^\theta$. Pick $y\in L\setminus\{x\}$. Then $y\perp y^\theta$, so $y^\theta$ is collinear to both $x$ and $y$, hence to $L$. Since there are at least two choices for $y$, we conclude that $L$ and $L^\theta$ are symplectic or opposite. It follows that $\theta$ is a polar $\{1,2'\}$-kangaroo. 
\end{proof}

We now eliminate the possibility of trialities of buildings of type $\mathsf{D}_4$.

\begin{prop}\label{prop:trialities}
No triality of $\mathsf{D}_4(\K)$ is a polar $\{1,2'\}$-kangaroo. 
\end{prop}

\begin{proof}
We work in the \textit{oriflamme complex} of $\mathsf{D}_4(\K)$, consisting of the points, lines and two classes of maximal subspaces of a nondegenerate hyperbolic quadric of Witt index 4. The maximal simplices consist of a point contained in a line contained in a maximal singular subspace of each type. A triality of $\mathsf{D}_4(\K)$ is then an automorphism of the oriflamme complex permuting the types of points and maximal singular subspaces in a cycle of length 3. 

Suppose that $\theta$ is a triality. By \cite[Main Result~2.2]{Mal:14}, we know that, if $\theta$ fixes some line, then it is conjugate to the standard triality of type $I_{\id}$ producing the split Cayley hexagon. However, we now show that such a triality is not a polar $\{1,2'\}$-kangaroo (we refer to \cite{Mal:14} for notation and background).  

Let $p$ be an absolute point, that is, $p\in \pi_p:=p^\theta\cap p^{\theta^2}$. Each line through $p$ in $\pi_p$ is fixed. Select a line $L$ through $p$ not containing any further absolute point (besides $p$) and not contained in $p^\theta\cup p^{\theta^2}$. Then $L^\theta$ is contained in $p^\theta$, but does not contain $p$. Hence it intersects $\pi_p$ in some point $x$. If $L$ and $L^\theta$ are not special, then they are symplectic, hence contained in a singular subspace $U$. Then $U$, which is of the type of $p^{\theta^2}$ containing $L$ and $px$, is mapped onto the point $L^\theta\cap (px)^\theta=L^\theta\cap px=\{x\}$. Hence $x$ and $U$ are absolute. Also, $x^\theta=U^{\theta^2}$ is generated by $L^{\theta^2} and (L^\theta)^{\theta^2}=L$. Since $x$ is absolute, all points of the plane $U\cap U^{\theta^2}$ are absolute, by the properties of trialities of type $I_{\id}$. But $L\subseteq U\cap U^{\theta^2}$, contradicting the choice of $L$. 

Thus we may assume that $\theta$ does not fix any line. Also, not all lines are mapped onto opposite lines (as such an automorphism is necessarily anisotropic and type preserving, by \cite[Theorem~1.3]{DPV:13}). But not all lines can be mapped onto symplectic ones either as this would mean the triality is domestic and then, by \cite[Main Result~2.1]{Mal:14} we are in the previous case again, a contradiction.  Hence there exists a line $L$ mapped to a symplectic one, and a line $M$ adjacent to $L$ mapped to an opposite. Set $p=L\cap M$.   By replacing $\theta$ with its inverse if needed, we may assume that $L^\theta$ is disjoint from but collinear to $L$. Let $U$ be the maximal singular subspace spanned by $L$ and $L^\theta$. By possibly redefining $(L,M)$ by $(L^\theta,M^\theta)$, we may assume that $U$ is mapped onto a point. 
The image $p^\theta$ intersects $U$ in a plane $\pi$ through $L^\theta$.  

Hence all lines of $U^{\theta^{-1}}$ through $p$ are mapped onto all lines  of $\pi$. 
Note that $U^{\theta^{-1}}\cap U$ contains $L$ and hence is a plane $\alpha$. Let $\alpha$ intersect $L^\theta$ in $x$. Select $y\in L^\theta\setminus\{x\}$ and such that the image of the planar line pencil in $\alpha$ through $p$ is not mapped onto the pencil in $\pi$ through $y$. Then the inverse image of the line pencil in $\pi$ through $y$ is a planar line pencil in $U^{\theta^{-1}}$ with only $L$ in~$U$. Hence the other lines of that pencil, being collinear with $\alpha$, cannot be collinear with their image. But they contain a point collinear to these images, namely $p$. So we obtain special pairs, a contradiction.
\end{proof}

\textbf{Digression---}We take the opportunity here to improve on Proposition~3.8 of \cite{NPVV}. This proposition states a few equivalent conditions for a collineation of a hyperbolic polar space of rank $n$ to be a polar $\{0,1,2'\}$-kangaroo. The weakest one mentioned there reads as follows. 
\begin{itemize}
\item[(Int($k$))] \emph{There exists $k$, $0\leq k\leq n-1$, such that for every maximal singular subspace $M$, the subspace $M\cap M^\theta$ is $k$-dimensional and globally fixed by $\theta$.}
\end{itemize}
Here, we will prove the following improvement.
\begin{prop}
Let $\theta$ be a collineation of a hyperbolic polar space of rank $n$. Then $\theta$ is a polar $\{0,1,2'\}$-kangaroo if, and only if, there  exists $k$, $0\leq k\leq n-1$, such that for every maximal singular subspace $M$, the subspace $M\cap M^\theta$ is $k$-dimensional. 
\end{prop}
\begin{proof}
Clearly, we only have to prove the ``if'' statement. So, we assume that there exists $k$, $0\leq k\leq n-1$, such that for every maximal singular subspace $M$, the subspace $M\cap M^\theta$ is $k$-dimensional. In view of (Int($k$)), it suffices to show that, for given such $M$, the subspace $U:=M\cap M^\theta$ is stabilised. We claim that $\theta$ is $\ell$-domestic for every $\ell\geq n-k-1$. Indeed, let $W$ be a singular subspaces of dimension $\ell$ and suppose for a contradiction that $W$ is mapped onto an opposite. Let $M$ be a maximal singular subspace containing $W$, then $M^\theta\cap M$ is disjoint from $W$ as $W$ does not contain any point collinear to each point of $W^\theta$. Hence $\dim M\cap M^\theta\leq n-1-\ell-1\leq k-1$, a contradiction to our assumption. Now, if $k\geq 1$, then $\theta$ is $(n-k)$-domestic and $(n-k-1)$-domestic, which implies by \cite[Theorem~6.1]{TTVM} that every singular subspace of dimension $n-k-1$ contains a fixed point. Let $S$ be such a subspace in $M$ intersecting $M^\theta$ in just one point $x$. Since $x$ is the only point of $S$ in $M^\theta$, it is the only point in $S^\theta$ and hence it has to be fixed.  We conclude that $M\cap M^\theta$ is pointwise fixing. 

Now let $k=0$. Then \cite[Lemma~2.1]{PVMclass} implies that every maximal singular subspace $M$ has a fixed point, which clearly coincides with $M\cap M^\theta$. 
\end{proof}

\section{Proof of the Main Result---Exceptional cases}\label{sec:exc}

In this section we prove Theorem~\ref{thm:A} for the exceptional simply laced types $\mathsf{E}_n$, $n=6,7,8$. We first deal with the case of type preserving automorphisms of buildings of type $\mathsf{E}_6$. 

\begin{prop}\label{prop:tpE6}
Theorem~\ref{thm:A} holds for type preserving automorphisms of $\mathsf{E}_6(\K)$. 
\end{prop}

\begin{proof}
By \cite[Main Result (a)]{NV} the fix structure of a type preserving automorphism is a naturally embedded quaternion or octonion Veronesean, which is an ideal Veronesean (in the sense of~\cite{NPV}) if and only if the collineation is a polar $\{1,2'\}$-kangaroo. Thus by Theorem~\ref{thm:basictheorem} the uniclass property is equivalent to the polar $\{1,2'\}$-kangaroo property. 
\end{proof}

Henceforth we consider oppomorphisms of buildings of type $\mathsf{E}_n$, $n=6,7,8$.

\subsection{Polar $\{1,2'\}$-kangaroo oppomorphisms}

In this subsection we prove the following theorem, which severely restricts the possibilities for $\{1,2'\}$-kangaroos via the classification of domestic automorphisms in exceptional types given in \cite{NPVV,PVMexc,PVMexc4,Mal:12}. 

\begin{thm}\label{thm:En}
Let $\theta$ be a non-trivial polar $\{1,2'\}$-kangaroo of $\mathsf{E}_n(\K)$ with $n\in\{6,7,8\}$. 

\begin{compactenum}[$(i)$]
\item If $n=6$ then $\theta$ is either anisotropic or domestic. 
\item If $n\in\{7,8\}$ then $\theta$ is either anisotropic, or $\theta$ is domestic and does not fix a chamber. 
\end{compactenum}
\end{thm}

Theorem~\ref{thm:En} follows immediately from Propositions~\ref{prop:id} and~\ref{opp2'} below. We begin by proving that each $\{1\}$-kangaroo automorphism of $\mathsf{E}_n(\K)$ with $n\in\{7,8\}$ which fixes a chamber, is the identity (see Proposition~\ref{prop:id}). Then we show that each non-domestic $\{2'\}$-kangaroo oppomorphism of $\mathsf{E}_n(\K)$ with $n\in\{6,7,8\}$ is anisotropic (see Proposition~\ref{opp2'}). These two results are dual to each other in the sense that distance~$2'$ can be viewed as codistance~$1$, and indeed the proofs can be viewed as being dual to each other. We note that one of the crucial facts that makes our proof work is the fact that the symps are of hyperbolic type. This is why a similar result for type $\mathsf{F}_4$ cannot be proved this way (and indeed we conjecture that the corresponding results are not true in this case).

\begin{prop}\label{prop:id}
Let $\theta$ be a collineation of a long root subgroup geometry $\Delta$ isomorphic to either $\mathsf{E_{7,1}}(\K)$ or $\mathsf{E_{8,8}}(\K)$ mapping no point to a distinct collinear one. Suppose that $\theta$ stabilises a symp $\xi$ and a maximal singular subspace $U\subseteq\xi$. Then $\theta$ is the identity.  
\end{prop}

\begin{proof}
We first claim that $\xi$ is pointwise fixed. Since no point is mapped to a collinear one, the singular subspace $U$ is pointwise fixed. Let $U'$ be a maximal singular subspace of $\xi$ intersecting $U$ in a hyperplane of $U$. Then $U\cap U'$ is fixed and hence, since $U$ is also fixed, the subspace $U'$ is stabilised, and hence pointwise fixed by the same observation as above. Since every point of $\xi$ is contained in a maximal singular subspace intersecting $U$ in a hyperplane of $U$,  the claim follows. 

Next we claim that every symp intersecting $\xi$ in a maximal singular subspace, is pointwise fixed. Indeed, every point $p$ of such a symp is close to $\xi$, and hence contained in a unique maximal singular subspace $W$ of $\Delta$ intersecting $\xi$ in a singular subspace $V$. By uniqueness of $W$ and the fact that $V$ is fixed, $\theta$ fixes $p$, as otherwise it would be mapped onto a collinear distinct point. The claim follows. 

Now the graph with vertices the symps of $\Delta$ adjacent when intersecting in a maximal singular subspace is connected. Hence $\theta$ fixes all symps and is therefore the identity.  
\end{proof}

\begin{prop}\label{opp2'}
Let $\theta$ be an oppomorphism of either $\mathsf{E_{6,2}}(\K)$, or $\mathsf{E_{7,1}}(\K)$, or $\mathsf{E_{8,8}}(\K)$ mapping no point to a point at distance $2'$. Suppose that $\theta$ maps some symp $\xi$ to an opposite symp, and maps a maximal singular subspace $U\subseteq\xi$ of $\xi$ to an opposite. Then $\theta$ is anisotropic.  
\end{prop}

\begin{proof}
The proof is very similar to the previous one, taking into account the following Observation~($*$). \begin{itemize}\item[($*$)] \emph{Each point of a singular subspace is either opposite or special to any point of an opposite singular subspace.}\end{itemize}   We first  claim that every point of $\xi$ is mapped to an opposite point. Since no point is mapped to a special one, Observation~($*$) implies that every point of $U$ is mapped to an opposite. Let $U'$ be a maximal singular subspace of $\xi$ intersecting $U$ in a hyperplane of~$U$. Then $U\cap U'$ is mapped to an opposite (as each point of it is mapped to an opposite) and hence, since $U$ is mapped to an opposite, and $U'$ is unique with respect to $U'\cap U$ in~$\xi$, the subspace $U'$ is also mapped to an opposite.  Hence each point of $U'$ is mapped to an opposite (using Observation~($*$) again). Since every point of $\xi$ is contained in a maximal singular subspace intersecting $U$ in a hyperplane of $U$,  the claim follows. 

Next we claim that every symp intersecting $\xi$ in a maximal singular subspace, is mapped to an opposite symp. Indeed, every point $p$ of such a symp $\zeta$ is close to $\xi$, and hence contained in a unique maximal singular subspace $W$ of $\Delta$ intersecting $\xi$ in a singular subspace $V$. By uniqueness of $W$ and the fact that $V$ is mapped to an opposite, $\theta$ maps $W$ to an opposite, and Observation~($*$) again implies that $p$ is mapped onto an opposite. Hence every point of the symp $\zeta$ is mapped onto an opposite and so $\zeta$ is mapped to an opposite. The claim is proved. 

Now the assertion follows as before from the connectivity of the graph on symps, adjacent when intersecting in a maximal singular subspace. 
\end{proof}

Thus the proof of Theorem~\ref{thm:En} is complete. Combining this theorem with the classification of domestic automorphisms of buildings of type $\mathsf{E}_n$ we conclude that if $\theta$ is a non-trivial polar $\{1,2'\}$-kangaroo oppomorphism of $\mathsf{E}_n(\K)$, $n\in\{6,7,8\}$, then $\theta$ is either anisotropic, or is uncapped, or is one of precisely six special classes of automorphisms listed below (see \cite{PVMexc,Mal:12} for the $\mathsf{E}_6$ case, \cite[Theorem~1]{NPVV} for the $\mathsf{E}_7$ case, and \cite[Theorem~A]{PVMexc4} for the $\mathsf{E}_8$ case): 
\begin{compactenum}
\item[--] \textit{Class 1:} Symplectic polarities in type $\mathsf{E}_6$. 
\item[--] \textit{Class 2:} Collineations of $\mathsf{E}_{7,1}(\K)$ fixing a metasymplectic space $\mathsf{F}_{4,1}(\K,\LL)$ for some quadratic extension $\LL$ of $\K$. 
\item[--] \textit{Class 3:} Collineations of $\mathsf{E}_{7,7}(\K)$ fixing a dual polar space (an ideal dual polar Veronesian).
\item[--] \textit{Class 4:} Collineations of $\mathsf{E}_{8,8}(\K)$ fixing a metasymplectic space $\mathsf{F}_{4,1}(\K,\mathbb{H})$ for some quaternion division algebra $\mathbb{H}$ over $\K$, or an inseparable field extension $\mathbb{H}$ of $\K$ of degree~$4$. 
\item[--] \textit{Class 5:} Collineations pointwise fixing an equator geometry in $\mathsf{E}_{7,1}(\K)$
\item[--] \textit{Class 6:} Collineations pointwise fixing an equator geometry in $\mathsf{E}_{8,8}(\K)$. 
\end{compactenum}

Comparing this list with Theorem~\ref{thm:basictheorem} we see that the automorphisms in classes 1, 2, 3, and 5 are uniclass, while those in classes 4 and 6 are not. Thus, in order to complete the proof of Theorem~\ref{thm:A} for buildings of type $\mathsf{E}_n$, it remains to:
\begin{compactenum}[$(1)$]
\item eliminate the possibility of uncapped $\{1,2'\}$-kangaroo oppomorphisms;
\item show that the oppomorphisms in classes 1--4 are indeed $\{1,2'\}$-kangaroos;
\item show that the oppomorphisms in classes 5 and 6 are not $\{1,2'\}$-kangaroos. 
\end{compactenum}

Task (1) is straightforward:

\begin{prop} Polar $\{1,2'\}$-kangaroo oppomorphisms of $\mathsf{E}_n(\K)$, $n\in\{6,7,8\}$, are necessarily capped. 
\end{prop}

\begin{proof}
By Theorem~\ref{thm:En} every $\{1,2'\}$-kangaroo oppomorphism is domestic. Suppose $\theta$ is not capped. By \cref{uncapped}, there exists a line $L$ of the corresponding long root geometry which is mapped onto an opposite, but no point of $L$ is mapped onto an opposite. Then clearly each point of $L$ is mapped onto a point at distance $2'$, a contradiction.
\end{proof}

The remaining tasks (2) and (3) require further work, and are dealt with in the following subsections.

\subsection{Class 1: Symplectic polarities in $\mathsf{E_6}$}
 A \textit{symplectic polarity} of a building of type $\sE_6$ is a type interchanging automorphism of order $2$  whose fixed point structure is a building of type $\sF_4$ containing residues isomorphic to symplectic polar spaces (such an $\sF_4$ building is a \textit{standard split metasymplectic space}).
%

As required by task~(2) listed above, we claim that a symplectic polarity is a polar $\{1,2'\}$-kangaroo. First note the following (see \cite[Lemma~4.2(2)]{PVM:19a}):

\begin{lemma}\label{lem:obs}
A symplectic polarity of $\mathsf{E_6}(\K)$ does not map any vertex of type $2$ or $4$ to an opposite. Equivalently, a symplectic polarity induces a collineation in $\mathsf{E_{6,2}}(\K)$ mapping no point or line to an opposite. 
\end{lemma}

We now show that in any long root subgroup geometry, a point-domestic and line-domestic collineation is necessarily a (polar) $\{1,2'\}$-kangaroo. 

\begin{prop}\label{dom=kang}
If $\Gamma=(X,\cL)$ is a long root subgroup geometry, then every collineation $\theta$ of $\Gamma$ that is both point-domestic and line-domestic is a $\{1,2'\}$-kangaroo. 
\end{prop}

\begin{proof}
Suppose that $\theta$ maps a point $p$ to a collinear point $q=p^\theta$. In $\Res_\Gamma(p)$, we can find a point opposite both points corresponding to the respective lines $pq$ and $p^{\theta^{-1}}p$ (use Proposition~3.30 of \cite{Tits:74}). To that point corresponds a line $L$ through $p$. Pick $x\in L\setminus\{p\}$ arbitrarily. Then $\{x,q\}$ and $\{p,x^\theta\}$ are special pairs. \cref{specialspecial} implies that $x$ and $x^\theta$ are opposite, a contradiction.

Now suppose that $\theta$ maps a point $p$ to a special point $q=p^\theta$, that is, $\{p,p^\theta\}$ is a special pair. Set $\{r\}=p^\perp\cap q^\perp$. Again we find a line $L$ through $p$ locally opposite both $pr$ and $pr^{\theta^{-1}}$. Choosing $x\in L\setminus\{p\}$ arbitrarily, we again find that $\{x,r\}$ and $\{x^\theta,r\}$ are special pairs. This implies that $\{x,q\}$ and $\{x^\theta,p\}$ are opposite pairs. Consequently, $L$ and $L^\theta$ are opposite lines, a contradiction.    
\end{proof}

Combining Lemma~\ref{lem:obs} and Proposition~\ref{dom=kang} achieves Task (2) for symplectic polarities of $\mathsf{E}_6$ buildings. 

\begin{remark}
Even more generally, \cref{dom=kang} holds in all so-called \emph{hexagonic geometries}. Without going into details, we mention that these are essentially the Lie incidence geometries having exactly the same distance relations between points as long root subgroup geometries. Examples are line Grassmannians of polar spaces, and the geometries of type $\mathsf{F}_{4,4}$.  
\end{remark}

In summary, combined with (the proof of) Proposition~\ref{prop:tpE6}, we have shown the following characterisation. 
Moreover, comparing with Theorem~\ref{thm:basictheorem} we have completed the proof of Theorem~\ref{thm:A} for buildings of type $\mathsf{E}_6$. 

\begin{thm}\label{E6uni}
An non-trivial automorphism $\theta$ of $\mathsf{E_{6}}(\K)$, for some field $\K$, is a polar $\{1,2'\}$-kangaroo if, and only if, it is an anisotropic duality, a symplectic polarity, or a collineation pointwise fixing an ideal Veronesean.
\end{thm}

\subsection{Class 2: Fixing a metasymplectic space in $\mathsf{E_{7,1}}(\K)$}
The main property, proved in \cite{NPVV}, is the following. 
\begin{prop}[{\cite[Theorem~7.23]{NPVV}}]\label{nofixedpoints}\label{notsymplectic}
Let $\theta$ be an automorphism of the building $\mathsf{E_7}(\K)$. Then the following are equivalent.
\begin{compactenum}[$(1)$]
\item $\theta$ does not fix any chamber and has opposition diagram $\mathsf{E_{7;3}}$.
\item The fixed point structure of $\theta$ induced in $\mathsf{E_{7,1}}(\K)$ is a fully isometrically embedded metasymplectic space $\mathsf{F_{4,1}}(\K,\LL)$, for some quadratic extension $\LL$ of $\K$. 
\item The collineation induced in $\mathsf{E_{7,7}}(\K)$ has no fixed points and does not map any point to a symplectic one, but it maps at least one point to a collinear one.  
\end{compactenum}
\end{prop}
Concerning (2), an equivalent way of saying ``fully isometrically embedded'' is saying it is a subspace with induced geometry a metasymplectic space in which opposite points are also opposite in $\mathsf{E_{7,7}}(\K)$. The fix diagram is $\mathsf{E_{7;4}}$. The geometry $\mathsf{F_{4,1}}(\K,\LL)$ is the Lie incidence geometry naturally associated to a building of type $\mathsf{F_4}$ obtained from a building $\mathsf{E_6}(\LL)$ by Galois descent, or  to a building associated to a group of mixed type over the pair of fields $(\K,\LL)$ in characteristic $2$, where $\LL$ is an inseparable quadratic extension of $\K$ (so $\dim_{\K}\LL=2$).

Now let $\theta$ be an automorphism of $\mathsf{E_7}(\K)$ with opposition diagram $\mathsf{E_{7;3}}$ and fix diagram $\mathsf{E_{7;4}}$. In \cite{NPVV}, the action of $\theta$ on $\mathsf{E_{7,7}}(\K)$ is examined. Then \cref{nofixedpoints}(3) tells us that the displacement of the points is highly restricted. But also the displacement of the symps is highly restricted. Indeed, Lemmas~7.7 and~7.16 of \cite{NPVV} say the following. 

\begin{lemma}\label{notspecial}\label{nottoadjacent}
Let $\xi$ be an arbitrary symp of $\mathsf{E_{7,7}}(\K)$. Then either $\xi$ is fixed, or $\xi\cap\xi^\theta$ is a line, or $\xi$ is opposite $\xi^\theta$. 
\end{lemma}

Translated to $\mathsf{E_{7,1}}(\K)$, this lemma says that $\theta$ is a polar $\{1,2'\}$-kangaroo, as required. 



\subsection{Class 3: Fixing a dual polar space in $\mathsf{E_{7,7}}(\K)$}
Let $\theta$ be an automorphism of $\mathsf{E_7}(\K)$ with opposition diagram $\mathsf{E_{7;4}}$ and fix diagram $\mathsf{E_{7;3}}$. Then the following assertions are shown in \cite{NPVV}, where again the action of $\theta$ on $\mathsf{E_{7,7}}(\K)$ is examined (see Corollaries~6.1 and~6.11, and Lemma~6.7 of \cite{NPVV}).

\begin{lemma}\label{evendist}\label{sympfixed} Let $\theta$ be an automorphism of $\mathsf{E_{7,7}}(\K)$ with opposition diagram $\mathsf{E_{7;4}}$ and fix diagram $\mathsf{E_{7;3}}$.
\begin{compactenum}[$(1)$]\item
A point is never mapped to a collinear one, nor to an opposite one. \item
The symp determined by a point $x$ mapped onto a symplectic one, and its image $x^\theta$, is stabilised. \item No symp is mapped onto an adjacent symp. \end{compactenum}
\end{lemma}
Lermma~6.9 of \cite{NPVV}  and its proof in \cite{NPVV} imply the following. 
\begin{lemma}\label{residver}\label{sympfixeduni}Let $\theta$ be an automorphism of $\mathsf{E_{7,7}}(\K)$ with opposition diagram $\mathsf{E_{7;4}}$ and fix diagram $\mathsf{E_{7;3}}$.
\begin{compactenum}[$(1)$] \item The collineation induced by $\theta$ in the residue of any fixed point, pointwise fixes an ideal 
Veronesean. 
\item The collineation induced by $\theta$ in a fixed symp pointwise fixes an ideal subspace which is a polar space of rank $2$. \item The fix diagram of $\theta$ is $\mathsf{E_{7;3}}$. \end{compactenum}
\end{lemma}
A substructure with the properties $(1)$ and $(2)$ of \cref{residver}, that is, a substructure consisting of a set $\cP$ of points, a set $\cL$ of lines and a set $\cS$ of symps of $\mathsf{E_{7,7}}(\K)$ such that \begin{compactenum}[$(i)$] \item the set of lines and symps in $\cL$ and $\cS$, respectively, incident with a point $p\in\cP$ defines an ideal Veronesean in the residue of $p$; \item the set of points and lines in $\cP$ and $\cL$, respectively, incident with a symp $\xi\in\cS$ defines an ideal subspace which is a polar space of rank $2$ in $\xi$,\end{compactenum} will be called an \emph{ideal dual polar Veronesean}.

\cref{evendist}(3) already shows that $\theta$ is a polar $\{1\}$-kangaroo. We now show that it is also a polar $\{2'\}$-kangaroo.

\begin{prop}
An automorphism $\theta$ of $\mathsf{E_7}(\K)$ with opposition diagram $\mathsf{E_{7;4}}$ and fix diagram $\mathsf{E_{7;3}}$ is a polar $\{1,2'\}$-kangaroo.
\end{prop}
\begin{proof}
It suffices to show that $\theta$ does not map a symp $\xi$ of $\mathsf{E_{7,7}}(\K)$
to a disjoint non-opposite symp. So, suppose for a contradiction that $\xi$ and $\xi^\theta$ are disjoint and not opposite. Then, by the last paragraph of \cref{e77}, there is a unique symp $\zeta$ intersecting $\xi$ in a maximal singular subspace $U$ and $\xi^\theta$ in a maximal singular subspace $U'$. Suppose $U^\theta=U'$. Let $x\in U$ be arbitrary. By \cref{evendist}(1), the point $x^\theta$ is symplectic to $x$ and so the symp determined by $x$ and $x^\theta$ is $\zeta$. Then \cref{evendist}(2) asserts that $\zeta$ is fixed, and \cref{residver}(2) implies that $U$ contains a fixed point, contradicting $\xi\cap\xi^\theta=\varnothing$. 

Hence there is some point of $U'$ which is the image of a point $y\in\xi\setminus U$.
Again, $y$ and $y^\theta$ are symplectic and determine a unique symp $\zeta'$, fixed under $\theta$. Clearly, $\zeta'\cap \xi$ is a maximal singular subspace $U_y$. As in the previous paragraph, $U_y$ contains a fixed point, leading to the same contradiction as before.

This proves the proposition. 
\end{proof}

\subsection{Class 4: Fixing a metasymplectic space in $\mathsf{E_{8,8}}(\K)$}
The main property, proved in \cite[Theorem~A]{PVMexc4}, Theorem~4.1, is the following.
\begin{prop}\label{domfixE8}
Let $\theta$ be an automorphism of the building $\mathsf{E_8}(\K)$. Then the following are equivalent.
\begin{compactenum}[$(1)$]
\item $\theta$ has opposition diagram $\mathsf{E_{8;4}}$ and the same fix diagram.
\item The fixed point structure of $\theta$ induced in $\mathsf{E_{8,8}}(\K)$ is a fully isometrically embedded metasymplectic space $\mathsf{F_{4,1}}(\K,\HH)$, for some quaternion division algebra $\HH$ of $\K$, or an inseparable field extension $\HH$ of degree $4$ of $\K$. 
\item  $\theta$ fixes a full subgeometry of $\mathsf{E_{8,8}}(\K)$ with fix diagram  $\mathsf{E_{8;4}}$.
\end{compactenum}
\end{prop}

The metasymplectic space $\mathsf{F_{4,1}}(\K,\HH)$ is the Lie incidence geometry naturally associated with a building of type $\mathsf{F_4}$ arising from $\mathsf{E_7}(\LL)$, with $\LL$ a separable quadratic extension of $\K$ contained in the quaternion algebra $\HH$ over $\K$, by Galois descent, or it is a building associated to a group of mixed type over the pair of fields $(\K,\HH)$ in characteristic $2$, where $\HH$ is an inseparable extension of $\K$ and $\dim_{\K}\HH=4$.

Now, Lemmas~4.4 and~4.5 of \cite{PVMexc4} precisely state that $\theta$ as in \cref{domfixE8} is a polar $\{1,2'\}$-kangaroo.

\subsection{Classes 5 and 6: Pointwise fixing an equator geometry in $\mathsf{E_{7,1}}(\K)$ or $\mathsf{E_{8,8}}(\K)$}
There are two remaining classes of domestic automorphisms of buildings of types $\mathsf{E_7}$ and $\mathsf{E_8}$ which do not fix any chamber. They have no fix diagram as they fix vertices of each type. We must show that these automorphisms are not $\{1,2'\}$-kangaroos. 

We treat them together, as they can be uniformly described in an elegant way using equator geometries. 

%
%


\begin{prop}
Let $\theta$ be a collineation of a long root subgroup geometry $\Gamma$ pointwise fixing an equator geometry $E(p,q)$ and acting fixed point freely on its set $\cI$ of poles. Then there exist points of $\Gamma$ mapped onto collinear or special ones.
\end{prop}
\begin{proof}
Without loss of generality we may suppose that $p^\theta=q$. Let $L$ be any line through $p$, and let $\xi$ be any symp through $L$. Let $\zeta$ be the unique symp through $q$ intersecting $\xi$ in some point $x\in E(p,q)$. Since $x$ is fixed, we note $\xi^\theta=\zeta$. On $L$ there is a unique point $u$ special to $q$; it is the unique point of $L$ collinear to $x$. It follows from \cref{pointlinesymp} that $u$ is collinear to the points of a unique line $M$ of $\zeta$ through $x$. Let $w$ be the unique point of $M$ collinear to $q$. Then $p\perp u\perp w\perp q$. Since $u$ and $q$ are special (by \cref{specialspecial}), the point $w$ is independent of the choice of $\xi$. It follows that the line $qw$ is contained in the image of every symp through $L$. Hence $L^\theta=qw$, implying $u^\theta\in qw$. But each point on $qw$ is either collinear or special to $u$, proving the assertion. 
\end{proof}

In summary, we have shown the following theorem.

\begin{thm}
A non-trivial automorphism $\theta$ of $\mathsf{E_7}(\K)$ or $\mathsf{E_8}(\K)$ is a polar $\{1,2'\}$-kangaroo if, and only if, exactly one of the following holds. 
\begin{compactenum}[$(1)$]
\item $\theta$ is anisotropic.
\item $\theta$ is domestic, its fix structure is a Weyl substructure and it has fix diagram one of $\mathsf{E_{7;3}}$, $\mathsf{E_{7;4}}$, or $\mathsf{E_{8;4}}$.
\end{compactenum} 
Consequently, $\theta$ is a polar $\{1,2'\}$-kangaroo if, and only if, it is uniclass.
\end{thm}

The proof of Theorem~\ref{thm:A} is now complete. 

\subsection{Counterexamples}\label{sec:counter}

We now provide examples illustrating that Theorem~\ref{thm:A} does not hold for non-simply laced spherical buildings. 

\begin{example}
Let $\Delta$ be a split building of type $\mathsf{F}_4$ over a field $\mathbb{K}$. If $\mathsf{char}\,\mathbb{K}\neq 2$ then the homologies $\theta=h_{\alpha}(-1)$ (with $\alpha$ any short root, and using standard notation in the Chevalley group; see for example \cite[Section~1.1]{PVMexc}) are $\{1,2'\}$-kangaroos that do not fix a Weyl substructure. Indeed, $\theta$ is a $\{1,2',3\}$-kangaroo (that is, each point of $\mathsf{F}_{4,1}(\mathbb{K})$ is either fixed, or mapped to distance~$2$). The proof of this fact can be extracted from the proof of \cite[Lemma~4.8]{PVMexc} where we the full displacement spectra of $\theta$ is obtained,  but it also follows directly from \cref{dom=kang} noting that the opposition diagram of $\theta$ is $\mathsf{F_{4;1}^4}$. Note that $\theta$ fixes a non-thick building with thick frame of type $\mathsf{B}_4$ (by \cite[Theorem~4.9]{PVMexc}), which is not a Weyl substructure. 

If $\mathsf{char}\,\mathbb{K}=2$ then there is no canonical choice of short and long roots, however once one class is declared long and the other short, the short root elations are $\{1,2'\}$-kangaroos (indeed, they are $\{1,2',3\}$-kangaroos) that do not fix a Weyl substructure. For the calculation of the displacement spectra of these automorphisms, see \cite[Theorem~2.1]{PVMexc}. 
\end{example}

\begin{example}
A split building $\Delta$ of type $\mathsf{B}_n$ is equivalent to a \emph{parabolic quadric} in some projective space of dimension $2n$, that is, a nondegenerate quadric of maximal Witt index $n$ in $\PG(2n,\K)$, for some field $\K$. The polar type corresponds to the lines, just like the case of type $\mathsf{D}_n$. The standard equation of such a quadric is, with respect to a suitable basis,
\[X_{-1}X_1+X_{-2}X_2+\cdots+X_{-n}X_n=X_0^2.\]
If $\kar\,\K\neq 2$, then the linear map defined on the coordinates by $X_i\mapsto X_i$, $i\in\{-n,-n+1,\ldots,-1,1,2,\ldots,n\}$ and $X_0\mapsto -X_0$ clearly induces an automorphism of    $\Delta$. But since the fix structure is hyperbolic quadric of Witt index $n$, it contains full chambers and hence cannot be a Weyl substructure. 

If $\kar\,\K=2$, then we consider a \emph{central elation}, that is, a collineation induced on $\Delta$ conjugate to the map defined on the coordinates by \[\begin{cases} X_1\mapsto X_1+X_{-1},\\ X_0\mapsto X_0+X_{-1},\\ X_{i}\mapsto X_i, & i\in\{-n,-n+1,\ldots,-1,2,3,\ldots,n\}.\end{cases}\]  
A central elation also fixes chambers and hence their fix structure is never a Weyl substructure. 

Now we claim that in both previous cases, lines are either fixed, or mapped to distance $2$ (that is, a line $L$ and its image $L'$ either coincide, or $L\cap L'$ is a point and $L$ and $L'$ are not contained in a common plane). Indeed, in both cases the set of fixed points is a \emph{geometric hyperplane}, that is, a set of points intersecting each line $K$ in either $K$ itself, or in a unique point. The claim then follows from \cite[Lemma 3.5.1]{Lam-Mal:23}.

Hence, in both cases, the collineation is a $\{1,2',3\}$-kangaroo. 
\end{example}

In view of these examples, it is tempting to conjecture that in the general case distinct from $\mathsf{C}_n$, a collineation preserving the polar type is uniclass if, and only if, it is a polar $\{1,2'\}$-kangaroo which is not a polar $3$-kangaroo.


\section{Another application of \cref{opp2'}}\label{sec:4}
We conclude with a characterisation of the root elations as polar kangaroos. The proof in the exceptional cases uses \cref{opp2'}. The polar kangaroo that we consider is the polar $\{2,2'\}$-kangaroo. In the exceptional cases, it really characterises root elations, whereas in the classical cases some other examples turn up. In fact, we can even allow all polar spaces (hence all buildings of types $\mathsf{B}_n$) and obtain a complete classification. Rather surprisingly, we can also include the case of buildings of type $\mathsf{F}_4$. 

\subsection{The exceptional simply laced cases}  
\begin{prop}\label{prop:exceptional22'}
A polar $\{2,2'\}$-kangaroo $\theta$ of $\mathsf{E_6}(\K)$, $\mathsf{E_7}(\K)$ or $\mathsf{E_8}(\K)$ is either anisotropic or a (central) root elation. 
\end{prop}

\begin{proof}
By \cite[Main Result~2]{NV}, a type preserving polar $\{2,2'\}$-kangaroo of $\mathsf{E_6}(\K)$ is a (central) elation. Hence from now on we may assume that $\theta$ is an oppomorphism. We consider the action of $\theta$ on the corresponding long root subgroup geometry. 
 
First suppose that $\theta$ is symp-domestic, that is, $\theta$ does not map any symp of the long root subgroup geometry onto an opposite one. Then Theorem~1 of \cite{PVM:19a} and Corollary~3 of \cite{PVM:19b} imply that either $\theta$ is trivial, or the opposition diagram of $\theta$ is one of

\[
%
\sE_{7;1}=\begin{tikzpicture}[scale=0.5,baseline=-0.5ex]
\node at (0,0.8) {};
\node at (0,-0.8) {};
\node [inner sep=0.8pt,outer sep=0.8pt] at (-2,0) (1) {$\bullet$};
\node [inner sep=0.8pt,outer sep=0.8pt] at (-1,0) (3) {$\bullet$};
\node [inner sep=0.8pt,outer sep=0.8pt] at (0,0) (4) {$\bullet$};
\node [inner sep=0.8pt,outer sep=0.8pt] at (1,0) (5) {$\bullet$};
\node [inner sep=0.8pt,outer sep=0.8pt] at (2,0) (6) {$\bullet$};
\node [inner sep=0.8pt,outer sep=0.8pt] at (3,0) (7) {$\bullet$};
\node [inner sep=0.8pt,outer sep=0.8pt] at (0,-1) (2) {$\bullet$};
\draw (-2,0)--(3,0);
\draw (0,0)--(0,-1);
\draw [line width=0.5pt,line cap=round,rounded corners] (1.north west)  rectangle (1.south east);
\end{tikzpicture} \hspace{.5cm} \mbox{or}\hspace{.5cm}
\sE_{8;1}=\begin{tikzpicture}[scale=0.5,baseline=-0.5ex]
\node at (0,0.8) {};
\node at (0,-0.8) {};
\node [inner sep=0.8pt,outer sep=0.8pt] at (-2,0) (1) {$\bullet$};
\node [inner sep=0.8pt,outer sep=0.8pt] at (-1,0) (3) {$\bullet$};
\node [inner sep=0.8pt,outer sep=0.8pt] at (0,0) (4) {$\bullet$};
\node [inner sep=0.8pt,outer sep=0.8pt] at (1,0) (5) {$\bullet$};
\node [inner sep=0.8pt,outer sep=0.8pt] at (2,0) (6) {$\bullet$};
\node [inner sep=0.8pt,outer sep=0.8pt] at (3,0) (7) {$\bullet$};
\node [inner sep=0.8pt,outer sep=0.8pt] at (4,0) (8) {$\bullet$};
\node [inner sep=0.8pt,outer sep=0.8pt] at (0,-1) (2) {$\bullet$};
\draw (-2,0)--(4,0);
\draw (0,0)--(0,-1);
\draw [line width=0.5pt,line cap=round,rounded corners] (8.north west)  rectangle (8.south east);
\phantom{\draw [line width=0.5pt,line cap=round,rounded corners] (1.north west)  rectangle (1.south east);}
\end{tikzpicture}.\]
Now Theorem 1 of \cite{PVMexc} implies that $\theta$ is a central long root elation. 

Secondly, suppose that some symp $\xi$ is mapped onto an opposite symp $\xi^\theta$. Let $p\in\xi$ be arbitrary. Since $\xi^\theta$ is opposite $\xi$, there exists some point of $\xi^\theta$ opposite $p$. By \cref{pointsymp}, the point $p^\theta$ is either symplectic to $p$, special to $p$, or opposite $p$. Since we assume that $\theta$ is a polar $\{2,2'\}$-kangaroo, the only possibility is that $p^\theta$ is opposite $p$. Hence $\theta$ maps each point of $\xi$ to an opposite point. It follows that every (maximal) singular subspace of $\xi$ is mapped onto an opposite. Now \cref{opp2'} implies that $\theta$ is anisotropic.   
\end{proof}

In the classical cases, we will classify polar $\{2,2'\}$-kangaroos without the parapolar terminology. We start with projective spaces.
\subsection{Projective spaces}\label{prosp} Let $\PG(n,\K)$ be a projective space of dimension $n\geq 2$ over the skew field $\K$.  We note that a two pairs $\{p,H\}$ and $\{p',H'\}$, where $p,p'$ are points and $H,H'$ are hyperplanes, are at distance $2$ or $2'$ in the long root subgroup geometry if, and only if, $H\neq H'$, $p\neq p'$ and $p\in H'$ or $p'\in H$. It follows that a collineation $\theta$ is a polar $\{2,2'\}$-kangaroo if, and only if, $\theta$ fixes every point in the intersection of a hyperplane and its image, as soon as these are distinct, and, dually, $\theta$ stabilises every hyperplane through each point $p$ and $p^\theta$, as soon as $p\neq p^\theta$. It follows easily that in such a case the line $pp^\theta$ is stabilised.

A \emph{Baer subplane} $B$ of $\PG(2,\K)$ is a subplane with the properties that each line of $\PG(2,\K)$ contains at least one point of $B$ and, dually, each point of $\PG(2,\K)$ is contained in at least one line of $B$.   A \emph{Baer collineation} $\theta$ of $\PG(2,\K)$ is a collineation whose fix structure is a Baer subplane. Typically, $B$ is a subplane isomorphic to $\PG(2,\FF)$, with $\FF$ a subfield of $\K$ over which $\K$ has dimension $2$, and after appropriately introducing coordinates,  $\theta$ acts on each coordinate as a skew field automorphism of $\K$ with fix subfield $\FF$. In case $\K$ is commutative, $\theta$ is always an involution. 

\begin{prop}
A non-trivial polar $\{2,2'\}$-kangaroo collineation $\theta$ of $\PG(n,\K)$, $n\geq 2$, is  a central collineation (elation or homology), or $n=2$ and $\theta$ is a Baer collineation.
\end{prop}

\begin{proof}
We may assume that $\theta$ is non-trivial. First suppose $n=2$. Then, since the line $pp^\theta$ is fixed as soon as $p\neq p^\theta$, \cite[Proposition~3.3]{PVMclass} implies that $\theta$ is either a central collineation or a Baer collineation. If $n\geq 3$, then the intersection $J$ of a non-fixed hyperplane with its image contains a line and is pointwise fixed. We select $x\in J$ and $y\notin J$. If all hyperplanes through $y$ not containing $x$ are stabilised, then all hyperplanes through $y$ are stabilised and $\theta$ is a central collineation. Hence we may assume that some hyperplane $H$ throiugh $y$ not through $x$ is stabilised. Then $H\cap H^\theta$ is fixed pointwise, but is diatinct from $J$. Consequently the subspace spanned by $J$ and $H\cap H^\theta$ is pointwise fixed, implying that $\theta$ is either the identity, or pointwise fixes a hyperplane. We again conclude that $\theta$ is a central collineation. 
\end{proof}

\begin{prop}
A  polar $\{2,2'\}$-kangaroo duality $\theta$ of $\PG(n,\K)$, $n\geq 2$, is anisotropic.  
\end{prop}

\begin{proof}
Suppose for a contradiction that $\theta$ is not anisotropic. Then some point $p$ is mapped onto a hyperplane $H$ containing $p$. There are two possibilities.
\begin{compactenum}[$(1)$]
\item Suppose $H^\theta=p$. Select $q\in H\setminus\{p\}$. Then $p=H^\theta\in q^\theta\neq H$. If $q\in q^\theta$, then $\{q,H\}$ is symplectic to $\{H^\theta,q^\theta\}=\{p,q^\theta\}$. If $q\notin q^\theta$, then $\{q,H\}$ is special to $\{p,q^\theta\}$. Both are contradictions.
\item Suppose $H^{\theta}\neq p$. Then $H^{\theta^2}\neq H$ and an arbitrary point $q\in pp^{\theta^2}\setminus\{p,p^{\theta^2}\}$  is mapped onto a hyperplane through $H\cap H^{\theta^2}$. The pair $\{q,H\}$ is mapped onto the pair $\{p^{\theta^2},q^\theta\}$. Since $q\neq p^{\theta^2}$, $H\neq q^\theta$ and $p^{\theta^2}\in H$ these two pairs are either symplectic or special, a contradiction.
\end{compactenum} 
We conclude that $\theta$ is anisotropic after all. 
\end{proof}

\subsection{Polar spaces}\label{linewise}
A polar $\{2,2'\}$-kangaroo of a building of type $\mathsf{D}_n$, $n\geq 4$, over the field $\K$, is a collineation of the corresponding polar space $\mathsf{D}_{n,1}(\K)$ mapping each line either to a coplanar one, or to an opposite. We take a broader perspective and look at collineations of all polar spaces of rank at least 3 with that property. We call such collineations \emph{linewise $\{2,2'\}$-kangaroos}.

We reduce the classification of linewise $\{2,2'\}$-kangaroos to the classification of certain domestic collineations.

In the next proof we freely use basic results on polar spaces (see \cite[Chapter~1]{Mal:24}). 

\begin{prop}\label{prop:22'polar}
Let $\theta$ be a non-trivial collineation of a polar space of rank at least~$3$. Then $\theta$ is a linewise $\{2,2'\}$-kangaroo if, and only if, either it is anisotropic, or among all singular subspaces only lines are mapped to opposites. 
\end{prop} 

\begin{proof}
First let $\theta$ be a non-trivial linewise $\{2,2'\}$-kangaroo of a polar space of rank at least~3. In view of the classification of opposition diagrams in \cite{PVM:19a,PVM:19b}, it suffices to show that, if $\theta$ is not anisotropic, then is is $0$-domestic and $3$-domestic, that is, it does not map any point to to an opposite and it does not map any singular $3$-space to an opposite one.

Suppose that $\theta$ maps some point $p$ to an opposite point $p^\theta$. Lines through $p$ are never coplanar with lines through $p'$, hence each line through $p$ is mapped onto an opposite line through $p^\theta$. Consequently, each plane through $p$ is mapped onto an opposite plane through $p^\theta$. Suppose for a contradiction that some point $x\perp p$ is mapped onto a collinear point $x^\theta$. Let $\pi$ be any plane through $p$ and $x$. Since $\pi$ and $\pi^\theta$ are opposite, there is a unique line $L$ in $\pi$ all of whose points are collinear to $x^\theta$. Note that $x\in L$. Then the image $L^\theta$ contains $x^\theta$ and is contained in $\pi^\theta$. Since $\pi$ and $\pi^\theta$ are opposite, they are disjoint, and so $L$ and $L^\theta$are not coplanar. Since $x^\theta\in L^\theta$ is collinear to all points of $L$, the lines $L$ and $L^\theta$ are not opposite either. We conclude that $x^\theta$ is opposite $x$. 

Likewise, all points         collinear to any point in $p^\perp$ are mapped onto an opposite, which implies that $\theta$ is anisotropic. 

Now suppose that $\theta$ maps a singular $3$-space $\Sigma$ to an opposite. Suppose that $\theta$ is not anisotropic. Since $\theta$ is then point-domestic by the previous arguments, the map $\theta_\Sigma$ mapping a point $x$ of $\Sigma$ onto the line $\Sigma\cap(x^\theta)^\perp$ induces a symplectic polarity in $\Sigma$. Let $L$ be a fixed line for that symplectic polarity. Then $L$ and $L^\theta$ are contained in a singular $3$-space, hence not opposite, and are yet not coplanar (because $\Sigma$ and $\Sigma^\theta$ are disjoint). This contradiction shows that $\theta$ is $3$-domestic.

Secondly, let $\theta$ be a non-trivial collineation which is $0$-domestic and $3$-domestic. Suppose $L$ is a line such that $L$ and $L^\theta$ are neither coplanar, nor opposite. If $L$ and $L^\theta$ intersect in a point $p$, then, since lines have at least three points, we can find a point $q\in L\setminus \{p\}$ which is mapped onto a point $q^\theta\neq p$. The points $q$ and $q^\theta$ are opposite, a contradiction. Suppose $L$ is such that $L$ and $L^\theta$ are disjoint but contained in a common singular $3$-space $\Sigma$. Then $\Sigma^{\theta^{-1}}$ contains $L$ and we can find a singular $3$-space $\Sigma_0$ through $L$ such that $\Sigma_0\cap\Sigma=\Sigma_0\cap\Sigma^{\theta^{-1}}=L$ and no point of $\Sigma_0\setminus L$ is collinear to all points of either $\Sigma$ or $\Sigma^{\theta^{-1}}$ (noting that, if the polar space has type $\mathsf{D_4}$, then, by \cite[Corollary~4]{PVM:19a} and \cite[Theorem~1]{PVM:19b}, $0$-domesticity implies that $\theta$ is type preserving). Now one checks easily that $\Sigma_0$ and $\Sigma_0^\theta$ are opposite, a contradiction. Suppose at last that $L$ and $L^\theta$ are   disjoint and there exists a unique point $p\in L$ collinear to all points of $L^\theta$. Then there exists a unique point $q\in L^\theta$ collinear to all points of $L$. Each point $x\in L$ distinct from $p$ and $q^{\theta^{-1}}$ is mapped onto an opposite point. 
\end{proof}

So, a polar $\{2,2'\}$-kangaroo which is not anisotropic has opposition diagram $\mathsf{B}_{n,1}^2$, $n\geq 3$, or $\mathsf{D}_{n,1}^2$, $n\geq 4$ (the case $n=3$ is contained in \cref{prosp}). Such collineations are classified by \cite[Theorem~3]{PVMclass}. Without going into details, we mention that these are always axial collineations (hence long root elations), except if the polar space is symplectic and the characteristic of the underlying field is not $2$ (then we have certain homologies), or the polar space has rank $3$ (and then also collineations qualify which pointwise fix a substructure isomorphic to a polar space of rank 3 inducing a Baer subplane in each plane that contains at least two fixed points).   

\subsection{Metasymplectic spaces}
Metasymplectic spaces are the natural point-line geometries for the buildings of type $\mathsf{F_4}$. Up to now, they only appeared in this paper as fix structures in larger buildings. Since we only need properties of them in this last paragraph of the paper for a non-essential digression to our main results, we refer the reader to the existing literature concerning precise definitions and basic properties. Good sources are \cite{Coh:82,Shult,Tits:74} and the accepted manuscript \cite{Lam-Mal:23}. Let us just mention that metasymplectic spaces behave very much like long root subgroup geometries in that they are parapolar spaces for which two different points are either collinear, special, symplectic or opposite (and in the latter case they are at distance $3$ from each other in the collinearity graph). The symps are polar spaces of rank~$3$ and form the point set of the \emph{dual} metasymplectic space (where the lines are the sets of symps containing a given singular plane). In that dual space, symplectic points correspond to symps intersecting in just a point. 

Here is our classification result.

\begin{prop}\label{metasymp}
Let $\theta$ be a $\{2,2'\}$-kangaroo of a metasymplectic space. Then $\theta$ is either anisotropic, or symp-domestic. In the latter case, this means that the opposition diagram for $\theta$ is either $\mathsf{F_{4;1}^1}$ or $\mathsf{F_{4;1}^4}$. Conversely, every domestic collineation with such opposition diagram is a $\{2,2'\}$-kangaroo in the corresponding metasymplectic space (that is, of type $\mathsf{F_{4,1}}$ if the opposition diagram is $\mathsf{F_{4;1}^4}$, and $\mathsf{F_{4,4}}$ if the opposition diagram is $\mathsf{F_{4;1}^1}$). 
\end{prop}

\begin{proof}
First assume that $\theta$ is a $\{2,2'\}$-kangaroo of the metasymplectic space $\Gamma=(X,\cL)$. If $\theta$ does not map any symp to an opposite, the result follows from the classification of opposition diagrams in \cite{PVM:19a}.  

So, suppose some symp $\xi$ is mapped onto an opposite symp. We want to prove an analogue of \cref{opp2'} for metasympledtic spaces. However, in view of \cref{sec:counter}, such analogue is only possible if we bring in an additional assumption, and that is exactly the $2$-kangaroo assumption. Indeed, each point of $\xi$ is either symplectic, special or opposite each point of $\xi^\theta$. By the $\{2,2'\}$-kangaroo assumption, points of $\xi$ are mapped onto opposite points. Hitherto, this is similar to the proof of \cref{prop:exceptional22'}. But now we have to deviate, because we cannot use \cref{opp2'}, and its proof certainly does not extend to the $\mathsf{F_4}$ case as metasymplectic spaces do not contain maximal singular subspaces that are not contained in symps. 

Let $\zeta$ be a symp intersecting $\xi$ in a plane $\pi$. Suppose for a contradiction that $\zeta^\theta$ is not opposite $\zeta$. Since $\pi^\theta$ is opposite $\pi$, the only two possibilities are that $\zeta\cap\zeta^\theta$ is a unique point $x$, or $\zeta$ is disjoint from $\zeta^\theta$, but there exists a symp $\xi_0$ intersecting both in planes. First assume $\zeta\cap\zeta^\theta=\{x\}$. Any point $y\in\pi\cap x^\perp$ is mapped onto a point of $\zeta^\theta$, hence collinear or symplectic to $x$, which implies that $\{y,y^\theta\}$ is either collinear, symplectic or special, a contradiction. Hence $\zeta\cap\zeta^\theta=\varnothing$ and there exists a symp $\xi_0$ intersecting both in planes, say $\alpha$ and $\alpha'$, respectively. Now the only points of $\zeta^\theta$ not symplectic or special to some point of $\alpha$ are those of $\alpha'$. Hence each point of $\alpha$ is sent to a collinear point in $\alpha'$. But then the map $\alpha\to\alpha:x\mapsto (x^\theta)^\perp\cap\alpha$ is a duality of $\alpha$ each point of which is absolute, a contradiction (see for instance \cite[Lemma~3.2]{TTVM3}).

Hence we have shown that $\zeta$ is opposite $\zeta'$. Then again, all points of $\zeta$ are mapped onto opposite points, and continuing like this, by connectivity, $\theta$ is anisotropic.

Now suppose that $\theta$ is domestic with one of the opposition diagrams $\mathsf{F_{4;1}^1}$ or $\mathsf{F_{4;1}^4}$. We have to choose the duality class of the metasymplectic space such that $\theta$ does not map symps to opposites, and hence only maps points to opposites. By \cite[Lemma~5.2.1]{Lam-Mal:23} no pair $\{x,x^\theta\}$, $x\in X$, is special. Suppose now for a contradiction that some point $x\in X$ is mapped onto a symplectic point $x'$. By \cite[Lemma~5.2.2$(iv)$]{Lam-Mal:23}, the symp $\xi$ containing $x$ and $x'$ is stabilised. Select a symp $\zeta$ with $\zeta\cap\xi=\{x\}$. Then $\zeta^\theta\cap\xi=\{x'\}$ and since $\{x,x'\}$ is symplectic, properties of metasymplectic spaces imply that $\zeta$ and $\zeta^\theta$ are opposite, a contradiction. 
\end{proof}

There is now an explicit list of collineations of metasymplectic spaces with opposition diagram $\mathsf{F_{4;1}^1}$ or $\mathsf{F_{4;1}^4}$, see \cite[Main Result]{Lam-Mal:23}. They are either the long root elations, or two specific examples given by their fix structure: either the fix structure is a geometric hyperplane given by an extended equator geometry together with its tropics geometry in the dual of the long root subgroup geometry of a split building of type $\mathsf{F_4}$ (equivalently, the fix structure in building theoretic terms is a weak subbuilding with thick frame of type $\mathsf{B_4}$), or the fix structure is a subspace isomorphic to the long root subgroup geometry of a split building of type $\mathsf{F_4}$ and $\Gamma$ corresponds to the long root subgroup geometry of a building of absolute type $\mathsf{F_4}$ and relative type $\mathsf{E_6}$ (using the notation of \cite[Table~II]{Tits:66}, with Tits index $\mathsf{^2E_{6,4}^2}$), 

\begin{remark}
As a final remark, we make the following observation. The main purpose of the present paper was to provide an alternative characterisation for uniclass automorphisms of spherical buildings in the simply laced case. As a digression we could also classify $\{2,2'\}$-kangaroo collineations in the Lie incidence geometries that behave like long root subgroup geometries. In \cite[\S7.3]{PVMexc4}, we classified, also as a digression, type preserving automorphisms of spherical buildings whose displacement spectra is contained in the union of two conjugacy classes, among which is the trivial class (call these for now \emph{lower biclass automorphisms}). The list of these automorphisms is strikingly similar to the list of polar $\{2,2'\}$-kangaroo collineations obtained above. Indeed, in both cases, one reduces to opposition diagrams with only one orbit encircled, which is necessarily the polar node(s). By the special nature of the long root subgroup geometry of type $\mathsf{C}_n$ (not containing collinear or special pairs of points), the opposition diagram $\mathsf{B}_{n;1}^1$ does not turn up in our analysis above of $\{2,2'\}$-kangaroos, whereas it does turn up in the context of lower biclass automorphisms. This makes it reasonable to conjecture that there will be a very tight connection between \emph{upper biclass oppomorphisms} of spherical buildings (the displacement is contained in two conjugacy classes of the Weyl group among which is the longest word) and polar $\{1,2\}$-kangaroo automorphisms. 
\end{remark}

\end{document}